\documentclass{article}

\usepackage{graphicx,xcolor,textpos}
\usepackage{helvet}
\usepackage[margin = 1in]{geometry}
\usepackage[utf8]{inputenc}
\usepackage[english]{babel}
\usepackage[T1]{fontenc}
\usepackage{gensymb}
\usepackage{float}
\usepackage{caption}
\usepackage{subcaption}
\usepackage{amsmath, mathtools, mathrsfs}
\usepackage{amssymb}
\usepackage{amsthm}
\usepackage{enumitem}
\usepackage[makeroom]{cancel}
\usepackage{hyperref}
\usepackage{tikz}
\usepackage{tikz-cd}
\usepackage{blkarray}
\usepackage[framemethod = default]{mdframed}
\usepackage{multirow}
\usetikzlibrary{calc}
\usepackage{todonotes}
\usepackage{siunitx}

\newtheorem{theorem}{Theorem}[section]
\theoremstyle{definition} 
\newtheorem{definition}[theorem]{Definition}
\theoremstyle{plain}
\newtheorem{prop}[theorem]{Proposition}
\theoremstyle{definition} 
\newtheorem{example}[theorem]{Example}
\theoremstyle{plain}
\newtheorem{lemma}[theorem]{Lemma}
\newtheorem{cor}[theorem]{Corollary}

\theoremstyle{definition}
\newtheorem{remark}[theorem]{Remark}
\theoremstyle{plain}
\newtheorem{theoremintro}{Theorem}[section]

\newtheorem{corintro}[theoremintro]{Corollary}

\DeclareMathOperator{\GL}{GL}

\DeclareMathOperator{\Aut}{Aut}

\DeclareMathOperator{\Id}{Id}

\DeclareMathOperator{\Gal}{Gal}

\DeclareMathOperator{\Perm}{Perm}

\DeclareMathOperator{\spn}{span}

\DeclareMathOperator{\Stab}{Stab}
\DeclareMathOperator{\Orb}{Orb}
\DeclareMathOperator{\LE}{LE}

\DeclareMathOperator{\supp}{supp}

\DeclareMathOperator{\std}{std}
\DeclareMathOperator{\End}{End}

\newcommand{\Z}{\mathbb{Z}}
\newcommand{\R}{\mathbb{R}}
\newcommand{\C}{\mathbb{C}}
\newcommand{\Q}{\mathbb{Q}}
\newcommand{\Qbar}{\overline{\mathbb{Q}}}
\newcommand{\N}{\mathbb{N}}
\newcommand{\G}{\mathcal{G}}

\newcommand{\g}{\mathfrak{g}}

\newcommand{\n}{\mathfrak{n}}

\newcommand{\asi}[1]{\prescript{\sigma}{}{#1}}
\newcommand{\ata}[1]{\prescript{\tau}{}{#1}}
\newcommand{\E}[1]{\{ 1, \ldots, #1 \}}

\title{A characterization of Anosov rational forms in \\ nilpotent Lie algebras associated to graphs}

\author{Jonas Der\'e and Thomas Witdouck\thanks{The second author was supported by a PhD fellowship of the Research Fund - Flanders (FWO), Grant Number 1153122N.}}

\begin{document}
	\maketitle
	
	\begin{abstract}
		Anosov diffeomorphisms are an important class of dynamical systems with many peculiar properties. Ever since they were introduced in the sixties, it has been an open question which manifolds can admit such diffeomorphisms, where tori of dimension greater than or equal to two are the typical examples. It is conjectured that the only manifolds supporting an Anosov diffeomorphism are finitely covered by a nilmanifold, a type of manifold closely related to rational nilpotent Lie algebras. 
		
		In this paper, we study the existence of Anosov diffeomorphisms for a large class of these nilpotent Lie algebras, namely the ones that can be realized as a rational form in a Lie algebra associated to a graph. From a given simple undirected graph, one can construct a complex $c$-step nilpotent Lie algebra, which in general contains different non-isomorphic rational forms, as described by the authors in previous work. We determine precisely which forms correspond to a nilmanifold admitting an Anosov diffeomorphism, leading to the first class of complex nilpotent Lie algebras having several non-isomorphic rational forms and for which all the ones that are Anosov are described. In doing so, we put a new perspective on certain classifications in low dimensions and correct a false result in the literature.
	\end{abstract}
	
	\section{Introduction}
	An Anosov diffeomorphism on a closed manifold $M$ is a diffeomorphism $f: M \to M$ that preserves a continuous splitting $TM = E^s \oplus E^u$ of the tangent bundle such that $df$ exponentially contracts the elements of $E^s$ and exponentially expands the elements in $E^u$. The easiest example of such a map is the so-called Arnold's cat map, i.e.~the map induced by the matrix $\begin{pmatrix}
		2 & 1 \\ 1 & 1
	\end{pmatrix}$ on the torus $\R^2/\Z^2$. In his seminal paper \cite{smal67-1}, S.~Smale introduced the first non-toral example of a manifold admitting an Anosov diffeomorphism, raising the question which manifolds can have an Anosov diffeomorphism. This new example was given on a nilmanifold $N/\Gamma$, i.e.~the compact quotient of a $1$-connected nilpotent Lie group $N$ by a lattice $\Gamma \subset N$. 

	It has been conjectured that every closed manifold admitting an Anosov diffeomorphism is, up to homeomorphism, finitely covered by a nilmanifold. Hence an important first step to understand the manifolds with an Anosov diffeomorphism is to study nilmanifolds. For tori, which are the compact quotients of the abelian Lie group $N = \R^n$, the problem is completely solved, namely an Anosov diffeomorphism exists if and only if $n > 1$. 
	
	Corresponding to every nilmanifold $N/\Gamma$, there is a (up to isomorphism) unique rational nilpotent Lie algebra $\n^\Q$, which is a rational form of the real Lie algebra $\n^\R$ corresponding to $N$. Conversely,  by the work of Mal'cev \cite{malc49-2}, every rational form of $\n^\R$ corresponds to a lattice $\Gamma$ of $N$ that is uniquely determined up to commensurability, i.e.~up to having an isomorphic subgroup of finite index. Moreover, by combining the results proven in \cite{deki99-1} and \cite{mann74-1} the existence of an Anosov diffeomorphism depends only on this rational Lie algebra $\n^\Q$. More specifically, $N/\Gamma$ admits an Anosov diffeomorphism if and only if $\n^\Q$ admits an automorphism which is both \textit{hyperbolic}, i.e. has no eigenvalues of absolute value 1, and \textit{integer-like}, i.e. has a characteristic polynomial with integer coefficients and constant term equal to $\pm 1$. Such an automorphism will also be called an \textit{Anosov automorphism} and a rational Lie algebra admitting one an \textit{Anosov Lie algebra}. In the remainder we hence focus on rational forms of real (or complex) $c$-step nilpotent Lie algebras to study Anosov diffeomorphisms on the corresponding nilmanifolds.

	There are several instances for which the existence of Anosov automorphisms is known. Free $c$-step nilpotent Lie algebras have a unique rational form, and this form admits an Anosov automorphism if and only if the number of generators is greater than or equal to $c+1$, see \cite{dv09-1}. A generalization in the $2$-step nilpotent case for one specific rational form of nilpotent Lie algebras associated to a simple undirected graph $\G$ is given in \cite{dm05-1}, we refer to Section \ref{sec:prelimOnLieAlgGraphs} for the exact definition of these Lie algebras. Note that the aforementioned example by S.~Smale is a rational form in such a real Lie algebra associated to a graph, but it does not fall under the main result of \cite{dm05-1}, see Example \ref{ex:twoCopiesHeisenberg}. There is also a classification of Anosov automorphisms on Lie algebras up to dimension $8$ in \cite{lw09-1}, containing two families of different rational forms in a fixed real Lie algebra of nilpotency class $2$, which are in both cases a Lie algebra associated to a graph. Later, this classification was slightly corrected in \cite{dere13-1}, where the close relation between Galois theory and Anosov diffeomorphisms was explored for the first time. Finally, there was an attempt to generalize the results in \cite{dm05-1} to the $c$-step nilpotent case in \cite{main12-1}, but as we will demonstrate below the main result of the latter paper is false.
	
	In conclusion, there are only sparse examples in the literature of real or even complex Lie algebras which contain more than one rational form and in which all forms with an Anosov automorphism are classified. In this paper we greatly extend the known results by fully characterizing the existence of Anosov diffeomorphisms on quotients of $c$-step nilpotent Lie groups associated to graphs. Our main result gives the first class of Lie algebras having distinct rational forms and with a general characterization of all Anosov rational forms. 

	Let $\G$ be a simple undirected graph. One can define an equivalence relation on the vertices by saying two vertices are equivalent if and only if their transposition (leaving all other vertices fixed) is a graph automorphism. The equivalence classes are called \textit{coherent components} and $\Lambda$ denotes the set of coherent components of $\G$. This allows one to define a quotient graph on $\Lambda$, written as $\overline{\G}$, which is a vertex weighted graph, possibly with loops. For any field $K \subset \C$ and integer $c > 1$, we write $\n_{\G, c}^K$ for the $c$-step nilpotent Lie algebra over $K$ associated to $\G$. These Lie algebras are defined in detail in section \ref{sec:prelimOnLieAlgGraphs}, together with the classification of their rational forms from \cite{dw22-1}. This classification gives for any injective group morphism $\rho:\Gal(L/\Q) \to \Aut(\overline{\G})$, with $L\subset \C$ a finite Galois extension of $\Q$, a rational form $\n_{\rho, c}^\Q$ of $\n_{\G, c}^\C$ and up to $\Q$-isomorphism, these are all the rational forms of $\n_{\G, c}^{\C}$.

	The main result of this paper uses the above description of all rational forms of $\n_{\G, c}^\C$ and provides an easy to check condition on the $\rho$-action of the Galois group on the quotient graph to determine whether the form is Anosov or not. In order to do so we define for any Galois extension $L/\Q$ with $L \subset \C$ and any morphism $\rho:\Gal(L/\Q) \to \Aut(\overline{\G})$ the following function on the coherent components
	\begin{equation}
		\label{eq:definitionz_rho}
		z_\rho:\Lambda \to \left\{\frac{1}{2}, \, 1\right\}: \lambda \mapsto
		\begin{cases}
			1 &\text{ if } \exists \sigma \in \Gal(L/\Q): \rho_{\tau \sigma}(\lambda) = \rho_{\sigma}(\lambda)\\
			\frac{1}{2} &\text{ else}
		\end{cases}
	\end{equation}
	where $\tau \in \Gal(L/\Q)$ denotes the complex conjugation automorphism (with the convention that it is the identity automorphism if $L \subset \R$). Note that $z_\rho$ is constant on $\rho$-orbits and that it takes the value $1$ on an orbit if and only if $\rho_\tau$ fixes a coherent component in that orbit. A set of coherent components will be called connected if the underlying set of vertices is connected in the graph (see Definition \ref{def:connectedGraph}).
	\begin{theoremintro}
		\label{thm:AnosovOrbitCondInjective}
		Let $\G$ be a graph, $\Lambda$ its set of coherent components, $L/\Q$ a finite Galois extension and $\rho: \Gal(L/\Q) \to \Aut(\overline{\G})$ an injective group morphism. The associated rational form $\n_{\rho, c}^\Q$ of $\n^{\C}_{\G, c}$ is Anosov if and only if for each non-empty connected set of coherent components $A \subset \Lambda$ such that $A \cup \rho_\tau(A)$ is $\rho$-invariant, it holds that
		\[c < \sum_{\lambda \in A \cup \rho_\tau(A)} z_\rho(\lambda) \cdot |\lambda|.\]
	\end{theoremintro}
	Note that the condition on the $\rho$-action only depends on how the orbits look like and which coherent components are fixed under the action of the complex conjugation automorphism, making it straight-forward to check the condition for concrete actions. The proof combines properties of hyperbolic algebraic units with the description of rational forms as given in Theorem \ref{thm:injectiveVersionClassificationRationalForms}. If $L \subset \R$, the condition for $\n_{\rho, c}^\Q$ to be Anosov reduces to: all non-empty connected $\rho$-invariant subsets $A \subset \Lambda$ must satisfy $c < \sum_{\lambda \in A} |\lambda|$. By applying the main result to the standard rational form $\n^\Q_\G$ we find that the main result of \cite{main06-1} is false, and we present a counterexample in Section \ref{sec:applications}

	In section \ref{sec:applications}, we demonstrate the ease to apply the aforementioned theorems to specific classes of graphs, a first of which is trees. We can formulate the result on the level of nilmanifolds as follows.
	
	\begin{corintro}
		\label{cor:AnosovNilmanifoldTree}
		Let $\G$ be a tree and $N_{\G, c}$ the associated $c$-step nilpotent simply connected Lie group. For any cocompact lattice $\Gamma \leq N_{\G,c}$, we have that the nilmanifold $N_{\G, c}/\Gamma$ does not admit an Anosov diffeomorphism.
	\end{corintro}
	
	The following natural class to consider are the cycle graphs, for which we found the following result on the level of nilmanifolds.
	
	\begin{corintro}
		\label{cor:AnosovNilmanifoldCycle}
		Let $\G$ be the cycle graph of order $n \geq 5$ and $N_{\G, c}$ the associated $c$-step nilpotent simply connected Lie group. There exists a cocompact lattice $\Gamma \leq N_{\G,c}$ such that the nilmanifold $N_{\G, c}/\Gamma$ admits an Anosov diffeomorphism if and only if $n > c$.
	\end{corintro}

	Finally, we illustrate how certain classifications in low dimensions are immediate from our main result.
	
	\section{Preliminaries}
	
	\subsection{Number fields and algebraic units with hyperbolic properties}
	\label{sec:numberFieldsAndAlgUnitsWithHypProp}
	
	All fields in the paper are assumed to be subfields of $\C$. A \textit{number field} is defined as a subfield $K \subset \C$ such that $K$ is a finite dimensional vector space over $\Q$, where we denote this dimension as $[K:\Q]$. It is well-known that a non-zero ring morphism between fields needs to be injective, and we will call such morphisms embeddings. Every number field $K \subset \C$ has exactly $[K:\Q]$ different embeddings $K \to \C$. If $\tau:\C \to \C$ denotes the complex conjugation automorphism, then an embedding $\sigma: K \to \C$ satisfies $\tau \circ \sigma = \sigma$ if and only if $\sigma(K) \subset \R$. We will call $\sigma$ a \textit{real embedding} if $\sigma(K) \subset \R$ and $\sigma, \tau \circ \sigma$ a \textit{conjugated pair of complex embeddings} if $\sigma(K) \nsubseteq \R$. If $s$ is the number of real embeddings and $t$ the number of conjugated pairs of complex embeddings, then $s + 2t = [K:\Q]$. We say a number field is \textit{totally real} if all its embeddings in $\C$ are real.
	
	An element of $\C$ is called \textit{algebraic over $\Q$} if it is a root of some polynomial with rational coefficients. The set of all elements algebraic over $\Q$ is a subfield of $\C$ and is equal to the algebraic closure of $\Q$, which we thus write as $\Qbar \subset \C$. Note that for any embedding $\sigma:K \to \C$ of a number field $K$, we have that $\sigma(K) \subset \Qbar$. By consequence we also have an embedding $\sigma:K \to \Qbar$. Conversely, every embedding of $K$ into $\Qbar$ gives an embedding of $K$ into $\C$ by composing it with the inclusion of $\Qbar$ in $\C$.
	
	For any extension $K \subset F$ of fields, we write $\Aut(F/K)$ for the group of field automorphisms of $F$ that fix every element of $K$. An extension $F/K$ will be called \textit{Galois} if $$\{ a \in L \mid \forall \sigma \in \Aut(L/K):  \sigma(a) = a\} = K.$$ In the case that the field extension $L/K$ is Galois, we write $\Gal(L/K)$ for $\Aut(L/K)$. A number field $K$ is a Galois extension of $\Q$ if and only if $\sigma(K) = K$ for any embedding $\sigma:K \to \C$. Note that $\C/\Q$ is not a Galois extension, but $\Qbar/\Q$ is. Therefore, in some cases, it will be more convenient to work with $\Qbar$ instead of $\C$.
	
	Given a number field, the following lemma shows the existence of a Galois extension of that field which satisfies a condition on the number of real and complex embeddings.
	\begin{lemma}
		\label{lem:existenceGaloisExtension}
		Let $K$ be a number field with $s$ real embeddings and $t$ conjugated pairs of complex embeddings in $\C$. For any positive integer $d \in \N$, there exists a Galois extension $F/K$ of degree $[F:K] = d$ such that $F$ has $d \cdot s$ real embeddings and $d \cdot t$ conjugated pairs of complex embeddings in $\C$.
	\end{lemma}
	\begin{proof}
		We can assume that all fields are subfields of $\C$. From \cite[Lemma 5.4.]{dw22-1}, it is clear that there exists a totally real Galois extension $L/\Q$ of degree $d$ such that $K \cap L = \Q$. Now take $F := KL$. Using Proposition 3.19 and Corollary 3.20 from \cite{miln22-1}, we find that $F/K$ is Galois and that
		\[[F:\Q] = [F:K][K:\Q] = [L:\Q][K:\Q] = d\cdot(s+2t).\]
		Let us write in general $\mathcal{M}_{K}$ for the set of embeddings from the number field $K$ to $\C$. Consider the map
		\[ g:\mathcal{M}_F \to \mathcal{M}_K \times \mathcal{M}_L: \sigma \mapsto (\sigma|_K, \sigma|_L). \]
		It is clear this map is injective since any embedding of $F$ in $\C$ is completely determined by its values on $K$ and $L$. Since the number of embeddings equals the degree of the number field, it follows from the above that $\mathcal{M}_F$ and $\mathcal{M}_K \times \mathcal{M}_L$ have the same cardinality. We thus have that $g$ is a bijection. If we let $\tau$ denote the complex conjugation automorphism on $\C$, we have for any $\sigma \in \mathcal{M}_F$ the equivalences
		\begin{align*}
			\tau \circ \sigma = \sigma &\Leftrightarrow g(\tau \circ \sigma) = g(\sigma)\\
			&\Leftrightarrow (\tau \circ \sigma|_K, \tau \circ \sigma|_L) = (\sigma|_K, \sigma|_L)\\
			&\Leftrightarrow (\tau \circ \sigma|_K, \sigma|_L) = (\sigma|_K, \sigma|_L)\\
			&\Leftrightarrow \tau \circ \sigma|_K = \sigma|_K.
		\end{align*}
		From this it follows that the number of real embeddings of $F$ into $\C$ is exactly equal to $d\cdot s$. As a consequence, the number of conjugated pairs of complex embeddings of $F$ into $\C$ must be equal to $d\cdot t$.
	\end{proof}
	
	The \textit{algebraic integers} of a number field $K$ are all elements of $K$ with minimal polynomial having integer coefficients. The set of algebraic integers of $K$ is closed under addition and multiplication and thus forms a ring. The units of this ring are called the \textit{algebraic units} of $K$ and are exactly those algebraic integers with minimal polynomial having constant term equal to $\pm 1$. Note that since an Anosov automorphism is integer-like (see introduction for definition), all of its eigenvalues are algebraic units. The group of algebraic units of a number field $K$ is written as $U_K$ and its structure has been described by Dirichlet's Unit Theorem (see \cite{st87-1}). The theorem states that if $K$ is a number field with $s$ real embeddings $\sigma_1, \ldots, \sigma_s:K \to \C$ and $t$ conjugated pairs of complex embeddings, where we list one of each pair as $\sigma_{s+1}, \ldots, \sigma_{s+t}:K \to \C$, then the map
	\begin{equation}
		\label{eq:DirichletsUnitThmMap}
		l:U_K \to \R^{s+t}: \xi \mapsto (\log|\sigma_1(\xi)|, \ldots, \log|\sigma_s(\xi)|, 2\log|\sigma_{s+1}(\xi)|, \ldots, 2\log|\sigma_{s+t}(\xi)|)
	\end{equation}
	has finite kernel and its image $l(U_K)$ is a cocompact lattice of the vector subspace
	\[ W_{s+t} := \{ (x_1, \ldots, x_{s+t}) \in \R^{s+t} \mid x_1 +\ldots + x_{s+t} = 0 \}. \]
	Using this description of $U_K$, we can prove the following lemma on the existence of algebraic units with hyperbolic properties. This lemma is a generalization of \cite[Proposition 3.6]{dd13-1} and the proof is very similar.
	
	\begin{lemma}
		\label{lem:existHypUnit}
		Let $K/\Q$ be a number field with $s$ real embeddings $\sigma_1, \ldots, \sigma_s:K \to \C$ and $t$ conjugated pairs of complex embeddings, where we list one of each pair as $\sigma_{s+1}, \ldots, \sigma_{s+t}:K \to \C$. Then for any $c \in \N$, there exists an algebraic unit $\xi \in K$ such that for any $e_1, \ldots,e_{s+t} \in \Z^{\geq 0}$ with $e_1 + \ldots + e_{s+t} \leq c$ it holds that
		\begin{equation} 
			\label{eq:chyperbolicity}
			\Bigl( \sigma_1(\xi)^{e_1} \cdot \ldots \cdot \sigma_{s+t}(\xi)^{e_{s+t}} = 1 \Bigr) \: \Rightarrow \: \Bigl( 2e_1 = \ldots = 2e_s = e_{s+1} = \ldots = e_{s+t} \Bigr).
		\end{equation}
	\end{lemma}
	\begin{proof}
		Let $l:U_K \to \R^{s+t}$ be the map as defined in (\ref{eq:DirichletsUnitThmMap}) for the number field $K$. From Dirichlet's Unit Theorem, we know that $l(U_K)$ is a cocompact lattice in the vector subspace $W_{s+t} \subset \R^{s+t}$. Now consider the following finite collection of hyperplanes in $\R^{s+t}$:
		\[ \mathcal{H} = \left\{ H \leftrightarrow \sum_{i=1}^s 2e_ix_i + \sum_{i=s+1}^{s+t} e_ix_i = 0 \quad \middle| \quad 
		\arraycolsep=1.4pt\def\arraystretch{1.5}
		\begin{array}{c}
		e_1, \ldots, e_{s+t} \in \Z^{\geq 0},\: \sum_{i=1}^{s+t} e_i \leq c \text{ and } \\
		\neg(2e_1 = \ldots = 2e_s = e_{s+1} = \ldots = e_{s+t})
		\end{array} 
		\right\}. \]
		Note that for any $H \in \mathcal{H}$ we have that $H \neq W_{s + t}$ and thus that $H \cap W_{s+t}$ is a $s+t-2$ dimensional vector subspace of $W_{s+t}$. Since a cocompact lattice in a real vector space can never be contained in a finite union of proper vector subspaces, it follows that there exists a $\xi \in U_K$ such that
		\[ l(\xi) \in W_{s+t} \setminus \left( \bigcup_{H \in \mathcal{H}} H \right). \]
		As one can check, $\xi$ satisfies the required property.
	\end{proof}
	
	\subsection{Lie algebras associated to graphs}
	\label{sec:prelimOnLieAlgGraphs}
	
	A graph is defined as a tuple $\G = (S, E)$ where $S$ is a finite set of vertices and $E$ is a subset of $\{ \{ \alpha, \beta \} \mid \alpha, \beta \in S, \alpha \neq \beta \}$, forming the edges. To this information, one can associate a Lie algebra in the following way. For any field $K \subset \C$, let $\mathfrak{f}^K(S)$ denote the free Lie algebra generated by the set $S$. Let $I_\G \subset \mathfrak{f}^K(S)$ denote the Lie ideal of $\mathfrak{f}^K(S)$ generated by the set of elements $\{[\alpha, \beta] \mid \alpha, \beta \in S, \{\alpha, \beta \} \notin E \}$. The Lie algebra $\g_{\G}^K$ is then defined as the quotient
	\[\g_{\G}^K = \mathfrak{f}^K(S)/I_\G.\]
	The Lie algebras obtained in this way are sometimes referred to as the \textit{free partially commutative Lie algebras}. For any Lie algebra $\g$, let $\gamma_i(\g)$ denote the $i$-th ideal in the lower central series of $\g$, defined inductively by $\gamma_1(\g) = \g$ and $\gamma_{i+1}(\g) = [\g, \gamma_{i}(\g)]$. One defines the $c$-step nilpotent Lie algebra associated to $\G$ over the field $K$ as
	\[\n_{\G,c}^K = \g_{\G}^K/\gamma_{c + 1}(\g_{\G}^K).\]

	\label{sec:groupOfGradedAutos}
	Let $V$ denote the vector subspace of $\n_{\G, c}^K$ spanned by the vertices $S$ and write $T_{\G, c} \leq \Aut(\n_{\G, c}^K)$ for the subgroup consisting of all automorphisms $f$ that satisfy $f(V) = V$. Note that we can identity the abelianization of $\n_{\G, c}^K$ with the vector space $V$ via the map $V \to \n_{\G, c}^K/[\n_{\G, c}^K, \n_{\G, c}^K]: v \mapsto v + [\n_{\G, c}^K, \n_{\G, c}^K]$. Every automorphism of $\n_{\G, c}^K$ induces an automorphism on the abelianization and thus via this identification also on the vector space $V$. This gives a group morphism
	\begin{equation}
		\label{eq:morphismp}
		p:\Aut(\n_{\G, c}^K) \to \GL(V).
	\end{equation}
	As one can check $p|_{T_{\G, c}}$ is injective and thus an isomorphism onto its image. It was proven in \cite[Proposition 2.2.]{dw22-1} that $p(T_{\G, c}) = p(\Aut(\n_{\G, c}^K))$ and that this image does not depend on the nilpotency class $c$. We will write $G = p(T_{\G, c}) = p(\Aut(\n_{\G, c}^K)) \leq \GL(V)$.  
	
	The group $G$ is a linear algebraic group and has been described in \cite{dm21-1} for fields $K \subset \C$. Let us briefly recall this description. For any vertex $\alpha \in S$, we define its \textit{open and closed neighbourhood} by
	\[ \Omega'(\alpha) = \{ \beta \in S \mid \{\alpha, \beta\} \in E \} \quad \quad \text{and} \quad \quad \Omega(\alpha) = \Omega(\alpha) \cup \{\alpha\}, \]
	respectively. A relation $\prec$ on $S$ is then defined by $\alpha \prec \beta\, \Leftrightarrow \, \Omega'(\alpha) \subset \Omega(\beta)$. Let us write $I$ for the identity map on $V$ and for any two vertices $\alpha, \beta \in S$, let us write $E_{\alpha \beta} \in \End(V)$ for the linear map which maps $\beta$ to $\alpha$ and any other vertex to zero. Let $M$ denote the unipotent subgroup of $\GL(V)$ generated by the elements $I + tE_{\alpha \beta}$ for which $t \in K$, $\alpha \prec \beta$ and $\alpha \nsucc \beta$.
	
	The relation $\prec$ also gives an equivalence relation $\sim$ on $S$ by $\alpha \sim \beta \, \Leftrightarrow \, \alpha \prec \beta \land \alpha \succ \beta$. Note that this equivalence relation can also be defined by saying $\alpha$ and $\beta$ are equivalent if and only if their transposition is a graph automorphism. We refer to the equivalence classes as the \textit{coherent components} of $\G$ and write the set of coherent components as $\Lambda = S/\sim$. 
	
	By modding out this equivalence relation, one gets a quotient graph $\overline{\G} = (\Lambda, \overline{E}, \Phi)$ where
	\[\overline{E} = \{\{\lambda, \mu\} \mid \lambda, \mu \in \Lambda, \exists \alpha \in \lambda, \beta \in \mu: \, \{\alpha, \beta\} \in E \}\}\]
	is the set of edges of the quotient graph and $\Psi:\Lambda \to \N: \lambda \mapsto |\lambda |$ is the weight function on the vertices. The automorphisms of the quotient graph are the bijections $\varphi:\Lambda \to \Lambda$ such that $\varphi(e) \in \overline{E}$ for any $e \in \overline{E}$ and $\Psi \circ \varphi = \Psi$. It is clear that there is a natural projection morphism $\Aut(\G) \to \Aut(\overline{\G}): \theta \mapsto \overline{\theta}$. If we order the vertices in each coherent component as $\lambda = \{ v_{\lambda 1}, v_{\lambda 2}, \ldots , v_{\lambda |\lambda|}\}$, we can also define a morphism
	\begin{equation}
		r: \Aut(\overline{\G}) \to \Aut(\G)
	\end{equation} by $r(\varphi)(v_{\lambda j}) = v_{\varphi(\lambda) j}$ for any $\varphi \in \Aut(\overline{\G})$, $\lambda \in \Lambda$ and $j \in \{1, \ldots  , |\lambda|\}$. This morphism satisfies $\overline{r(\varphi)} = \varphi$ for any $\varphi \in \Aut(\overline{\G})$ and $r(\varphi)|_\lambda = \Id_\lambda$ for any $\lambda \in \Lambda$ which is a fixed point of $\varphi$. We fix this ordering of the vertices and this morphism $r$ for the remainder of the paper.
	
	For any permutation $\theta:S \to S$, we write $P(\theta) \in \GL(V)$ for the linear map determined by $P(\theta)(\alpha) = \theta(\alpha)$ for any vertex $\alpha \in S$. This gives a morphism $P:\Aut(\G) \to \GL(V)$. Composing this map with the morphism $r$ from above, we get a morphism $\overline{P} = P \circ r : \Aut(\overline{\G}) \to \GL(V)$. For any $\lambda \in \Lambda$, write $V_\lambda$ for the vector subspace of $V$ spanned by the vertices in $\lambda$. We let $\GL(V_\lambda)$ denote the subgroup of $\GL(V)$ of all invertible linear maps which fix the vertices which do not lie in $\lambda$.
	
	As proven in \cite{dm21-1,dw22-1}, the group $G = p(T_{\G, c})$ is equal to
	\[ G =  M \cdot \left(\prod_{\lambda \in \Lambda} \GL(V_\lambda) \right) \cdot \overline{P}(\Aut(\overline{\G})). \]
	Moreover, the group can be written as an internal semi-direct product \[G \cong M \rtimes \left(\left(\prod_{\lambda \in \Lambda} \GL(V_\lambda) \right)  \rtimes \overline{P}(\Aut(\overline{\G}))\right)\]
	where $\left(\prod_{\lambda \in \Lambda} \GL(V_\lambda) \right)  \rtimes \overline{P}(\Aut(\overline{\G}))$ is reductive and $M$ is the unipotent radical of $G$. We thus also have a projection morphism
	\begin{equation*}
		\label{eq:morphismq}
		q:G \to \Aut(\overline{\G})
	\end{equation*}
	Composing this with the morphism $p$, we get the morphism
	\begin{equation}
		\label{eq:morphismpi}
		\pi := q \circ p: \Aut(\n_{\G, c}^K) \to \Aut(\overline{\G}).
	\end{equation}
	Since $p|_{T_{\G, c}}:T_{\G, c} \to G$ is a group isomorphism, it has a well-defined inverse. If we compose this inverse with the morphism $\overline{P} = P \circ r:\Aut(\overline{\G}) \to G$, we get an injective morphism
	\begin{equation}
		\label{eq:morphismi}
		i = (p|_{T_{\G, c}})^{-1} \circ \overline{P}:\Aut(\overline{\G}) \to T_{\G, c} \leq \Aut(\n_{\G, c}^K).
	\end{equation}
	
	\subsection{The rational forms of $\n_{\G, c}^\C$}
	\label{sec:rationalFormsOfLieAlgGraphs}
	\label{sec:GaloisActionOnLieAlg}
	If $L/K$ is an extension of subfields of $\C$, then there is a natural inclusion $\mathfrak{f}^K(S) \subset \mathfrak{f}^L(S)$ induced by the identity on $S$. This induces an inclusion $\n_{\G, c}^K \subset \n_{\G, c}^L$. In fact $\n_{\G, c}^K$ is a $K$-form of $\n_{\G, c}^L$ where the inclusion $\n_{\G, c}^K \subset \n_{\G, c}^L$ gives a natural isomorphism of Lie algebras associated to graphs
	\[ h:\n_{\G, c}^K \otimes_K L \stackrel{\cong}{\longrightarrow} \n_{\G, c}^L: v \otimes l \mapsto lv. \]
	If $L/K$ is a Galois extension, we also have a natural action of $\Gal(L/K)$ on $\n_{\G, c}^L$ by semi-linear maps defined by $\asi{h(v \otimes l)} = h(v \otimes \sigma(l))$ for any $\sigma \in \Gal(L/K)$, $v \in \n_{\G, c}^K$ and $l \in L$ and extending this additively to all of $\n_{\G, c}^L$. This action also gives an action of $\Gal(L/K)$ on $\Aut(\n_{\G, c}^L)$ by defining for any $f \in \Aut(\n_{\G, c}^L)$, $v \in \n_{\G, c}^L$ and $\sigma \in \Gal(L/K)$:
	\[ \left(\asi{f}\right)(v) = \asi{\left( f\left(\prescript{\sigma^{-1}}{}{v}\right) \right)}. \]
	Note that the subgroup $T_{\G, c} \leq \Aut(\n_{\G, c}^L)$ is invariant under this action and thus that we also get an induced action of $\Gal(L/K)$ on $G$ by using the isomorphism $p|_{T_{\G, c}}:T_{\G, c} \to G$. For these $\Gal(L/K)$ actions, the map $p:\Aut(\n_{\G, c}^L) \to G$ is $\Gal(L/K)$-equivariant, i.e. $p(\asi{f}) = \asi{p(f)}$ for any $f \in \Aut(\n_{\G, c}^L)$ and $\sigma \in \Gal(L/K)$.

	The rational forms of $\n_{\G, c}^\C$ can be described by the following theorem from \cite{dw22-1}.
	\begin{theorem}
		\label{thm:injectiveVersionClassificationRationalForms}
		Let $\G$ be a simple undirected graph and $\n_{\G, c}^\C$ the associated $c$-step nilpotent complex Lie algebra. Up to $\Q$-isomorphism, all rational forms of $\n_{\G,c}^\C$ are given by
		\[ \n_{\rho, c}^\Q = \{ v \in \n^L_{\G, c} \mid \forall \sigma \in \Gal(L/\Q): i(\rho_\sigma)(\asi{v}) = v \} \]
		where $L/\Q$ is a finite Galois extension and $\rho:\Gal(L/\Q) \to \Aut(\overline{\G})$ is an injective group morphism. If $K/\Q$ is another finite Galois extension with an injective group morphism $\eta:\Gal(K/\Q) \to \Aut(\overline{\G})$, then $\n_{\rho, c}^\Q$ and $\n_{\eta, c}^\Q$ are isomorphic if and only if $L = K$ and there exists a $\varphi \in \Aut(\overline{\G})$ such that $\varphi \rho(\sigma) \varphi^{-1} = \eta(\sigma)$ for all $\sigma \in \Gal(L/\Q)$. The rational forms of $\n^\R_{\G, c}$ are exactly those $\n_{\rho, c}^\Q$ for which $L \subset \R$.
	\end{theorem}
	
	As the eigenvalues of an Anosov diffeomorphism on the rational form $\n_{\rho, c}^\Q$ do not necessarily lie in the finite Galois extension $L$, we will also use the algebraic closure $\Qbar$ of $\Q$ in section \ref{sec:reductionToAlgInt}. Note that an injective morphisms $\rho:\Gal(L/\Q) \to \Aut(\overline{\G})$ can always be extended to a morphism with domain $\Gal(\Qbar/\Q)$ by precomposing it with the morphism $\Gal(\Qbar/\Q)\to \Gal(L/\Q):\sigma \mapsto \sigma|_L$. Let $\overline{\rho}:\Gal(\Qbar/\Q) \to \Aut(\overline{\G})$ denote this extended morphism. It is continuous with respect to the Krull topology on $\Gal(\Qbar/\Q)$ and the discrete topology on $\Aut(\overline{\G})$. These will always be the topologies considered on these groups. To this morphism $\overline{\rho}$ (and in fact to any continuous morphism from $\Gal(\Qbar/\Q)$ to $\Aut(\overline{\G})$), one can again associate a rational form as a subset of $\n_{\G, c}^{\Qbar}$ by
	\[ \n_{\overline{\rho}, c}^\Q = \{ v \in \n_{\G, c}^{\Qbar} \mid \forall \sigma \in \Gal(\Qbar/\Q): \, i(\overline{\rho}_\sigma)(\asi{v}) = v \}.\]
	This rational form is isomorphic to the original one $\n_{\rho, c}^\Q \subset \n_{\G, c}^L$. The following lemma will be useful in section \ref{sec:reductionToAlgInt} for finding automorphisms of rational forms.
	\begin{lemma}
		\label{lem:autosOfRatForm}
		Let $\rho:\Gal(\Qbar/\Q) \to \Aut(\overline{\G})$ be a continuous group morphism. Let $\n_{\rho, c}^\Q \subset \n_{\G, c}^{\Qbar}$ be the associated rational form. An automorphism $f \in \Aut(\n_{\G, c}^{\Qbar})$ satisfies $f(\n_{\rho, c}^\Q) = \n_{\rho, c}^\Q$ if and only if
		\[ i(\rho_\sigma^{-1}) \, f \,i(\rho_\sigma) = \asi{f} \]
		for all $\sigma \in \Gal(\Qbar/\Q)$.
	\end{lemma}
	\begin{proof}
		Let $f \in \Aut(\n_{\G, c}^{\Qbar})$. Note that the condition $f(\n_{\rho, c}^\Q) = \n_{\rho, c}^\Q$ is equivalent to $f(\n_{\rho, c}^\Q) \subset \n_{\rho, c}^\Q$. This follows from the fact that $f$ is also a $\Q$-linear bijection and the fact that $\n_{\rho, c}^\Q$ is finite dimensional. We thus have that $f(\n_{\rho, c}^\Q) = \n_{\rho, c}^\Q$ if and only if $\forall v \in \n_{\rho, c}^\Q, \, \sigma \in \Gal(\Qbar/\Q): i(\rho_\sigma)\left( \asi{\left(f(v)\right)} \right) = f(v)$. This last equality can be rewritten as
		\begin{align*}
			i(\rho_\sigma)\left( \asi{\left(f(v)\right)} \right) = f(v) &\Leftrightarrow (f^{-1} \, i(\rho_\sigma) \, \asi{f}) (\asi{v}) = v\\
			&\Leftrightarrow (f^{-1} \, i(\rho_\sigma) \, \asi{f} \, i(\rho_\sigma^{-1}) \, i(\rho_\sigma) )(\asi{v}) = v\\
			&\Leftrightarrow (f^{-1} \, i(\rho_\sigma) \, \asi{f} \, i(\rho_\sigma^{-1}))(v) = v\\
		\end{align*}
	Thus, we find that $f(\n_{\rho, c}^\Q) = \n_{\rho, c}^\Q$ if and only if the automorphism $f^{-1} \, i(\rho_\sigma) \, \asi{f} \, i(\rho_\sigma^{-1}) \in \Aut(\n_{\G, c}^{\Qbar})$ is the identity on the rational form $\n_{\rho, c}^\Q$ for any $\sigma \in \Gal(\Qbar/\Q)$. But since $\n_{\rho, c}^\Q$ is a form, it has a basis which is also a basis for $\n_{\G, c}^{\Qbar}$. By consequence we find that $f(\n_{\rho, c}^\Q) = \n_{\rho, c}^\Q$ if and only if $i(\rho_\sigma^{-1}) \, f \,i(\rho_\sigma) = \asi{f}$ for all $\sigma \in \Gal(\Qbar/\Q)$.
	\end{proof}

	\section{Reduction to algebraic integers}
	
	Before proving Theorem \ref{thm:AnosovOrbitCondInjective}, we first need a technical result which relates the existence of an Anosov automorphism to having certain algebraic units for every vertex of our graph. A first step is to determine the eigenvalues of certain automorphisms of $\n_{\G,c}^K$ depending on the structure of the graph. Later, we will apply this to relate Anosov automorphisms to the existence of certain algebraic units.
	
	\subsection{The eigenvalues of vertex diagonal automorphisms of $\n_{\G, c}^K$}
	\label{sec:eigenvaluesVertexDiagAutos}
	Let $\G = (S, E)$ be a graph, $K \subset \C$ a field and $\n_{\G, c}^K$ the associated $c$-step nilpotent Lie algebra over $K$. We say an automorphism $f \in \Aut(\n_{\G, c}^K)$ is \textit{vertex diagonal} if there exists a map $\Psi:S \to K$ such that $f(\alpha) = \Psi(\alpha)\alpha$ for any vertex $\alpha \in S$. In fact, for the Lie algebras $\n_{\G, c}^K$ it holds that for any map $\Psi:S \to K \setminus \{0\}$ there exists a unique vertex diagonal automorphism $f_\Psi \in \Aut(\n_{\G, c}^K)$ with $f_\Psi(\alpha) = \Psi(\alpha) \alpha$ for all $\alpha \in S$. We say that $f_\Psi$ is the vertex diagonal automorphism determined by $\Psi$.

	Given a map $\Psi: S \to K$ we will give in this section a description of all eigenvalues of the vertex-diagonal automorphism $f_{\Psi} \in \Aut(\n_{\G, c}^K)$. This description will depend on whether certain subsets of the graph are connected. Recall that $\{S_1, S_2,\ldots, S_n\}$ is a \textit{partition} of $S$ if all $S_i$ are non-empty, $S_i \cap S_j = \emptyset$ for any $i \neq j$ and $S = \bigcup_{i = 1}^n S_i$.
	
	\begin{definition}
		\label{def:connectedGraph}
		A graph $\G = (S, E)$ is called \textit{connected} if for any partition $\{S_1, S_2\}$ of $S$, there exist vertices $\alpha_1 \in S_1$ and $\alpha_2 \in S_2$ such that $\{\alpha_1 , \alpha_2\} \in E$. If a graph is not connected we simply say it is \textit{disconnected}. A subset of vertices $A \subset S$ will be called \textit{(dis)connected} if the subgraph spanned by $A$ is (dis)connected. A subset of coherent components $B \subset \Lambda$ will be called \textit{(dis)connected} if the union $\bigcup_{\lambda \in B} \lambda$ is (dis)connected.
	\end{definition}
	
	In what follows, we will use the notion of \textit{Lyndon elements} which is a generalization of the so called \textit{Lyndon words} to the partially commutative setting.  We refer to \cite{wade15-1} for a more detailed overview. This paper summarizes results from Droms, Duchamp, Krob,
	and Lalonde \cite{drom83-1} \cite{lalo91-1} \cite{dk92-1} \cite{kl93-1}. The reader should be careful since there are two different conventions when using graphs to define an algebraic object. We are using the convention that two non-adjacent vertices commute in the associated algebraic structure, as is common when considering Lie algebras, see \cite{dm05-1}. Note that under this convention, the vertices and edges of the graph form a basis of the Lie algebra in the $2$-step nilpotent case. However, the opposite convention is also often used in the literature when considering groups, as for example in \cite{wade15-1}.

	We denote by $W(S)$ the set of all words with letters in $S$, including the empty word which we write with $\emptyset$. For two words $w_1, w_2$, we write $w_1 \leftrightarrow w_2$ if there exist words $u_1, u_2$ and vertices $\alpha, \beta$ with $\{\alpha, \beta\} \notin E$ such that $w_1 = u_1 \alpha \beta u_2$ and $w_2 = u_1 \beta \alpha u_2$. We then define an equivalence relation $\sim$ on $W(S)$ by
	\[w_1 \sim w_2 \iff \exists k \in \N, u_1, \ldots ,u_k \in W(S): w_1 \leftrightarrow u_1 \leftrightarrow \ldots \leftrightarrow u_k \leftrightarrow w_2. \]
	The equivalence class of a word $w$ will be denoted by $\overline{w}$ and the set of all equivalence classes is denoted by $M := W(S) / \sim$. Given an order on $S$, we give $W(S)$ the lexicographical order. We then denote for any $m \in M$ by $\std(m)$ the maximal element of $m$ with respect to this lexicographical order on $W(S)$. This maximum exists because the equivalence class $m$ is a finite set.
	
	We define the \textit{length of a word} $w \in W(S)$ as the number of letters making up $w$ and denote this number with $|w|$. The \textit{weight of a word} $w \in W(S)$ is defined as the map $e_w:S \to \Z^{\geq 0}$ which assigns to $\alpha \in S$ the number of times it occurs in $w$. Note that these notions descend nicely to the equivalence classes in $M$ where the length of an element $m = \overline{w} \in M$ is defined as $|m| := |w|$ and the weight of $m$ is defined as $e_m := e_w$.
	
	To introduce the Lyndon words and Lyndon elements we also need the notions of conjugacy classes of words and primitive words. The \textit{conjugacy class} of a word $w \in W(S)$ is the set of all words $v \in W(S)$ such that there exist words $u_1, u_2 \in W(S)$ with $w = u_1 u_2$ and $v = u_2 u_1$. We say a word $w$ is \textit{primitive} if there exist no non-empty words $u_1, u_2 \in W(S)$ such that $w = u_1 u_2 = u_2 u_1$.
	
	\begin{definition}
		We say a word $w \in W(S)$ is a \textit{Lyndon word} if it is not the empty word, it is primitive and it is minimal in its conjugacy class. An element $m \in M$ is then called a \textit{Lyndon element} if $\std(m)$ is a Lyndon word. The set of Lyndon elements of $M$ is denoted by $\LE(M)$.
	\end{definition}
	
	Note that this is not the original definition of a Lyndon element, but we use the characterization from \cite[Theorem 5.12]{wade15-1}. In what follows we will determine the possible weights of Lyndon elements, thus giving useful information about the eigenvalues of vertex diagonal automorphisms of $\n_{\G,c}^L$. The following lemma will be crucial for finding Lyndon elements from connected subgraphs.
	
	\begin{lemma}
		\label{lem:existenceLyndonWord}
		Let $\G = (S, E)$ be a graph with an order on $S$ and $\alpha_1, \ldots, \alpha_k$ distinct elements of $S$ with $k \geq 2$. If $\{\alpha_1, \ldots, \alpha_k \}$ is connected, then there exists a permutation $\sigma \in S_k$ such that
		\[ \overline{\alpha_{\sigma(1)}^{e_1} \alpha_{\sigma(2)}^{e_2} \ldots \alpha_{ \sigma(k)}^{e_k}} \]
		is a Lyndon element with respect to the ordering on $S$ for any positive integers $e_1, \ldots, e_k > 0$ . 
	\end{lemma}
	\begin{proof}
		First we construct the permutation $\sigma$. Define $\sigma(1)$ by $\alpha_{\sigma(1)} = \min\{\alpha_1, \ldots, \alpha_k\}$. Then define the other images of $\sigma$ inductively by choosing for $\sigma(i)$ an element from $\{ 1, \ldots, k \} \setminus \{ \sigma(1), \ldots, \sigma(i-1) \}$ such that $\{\alpha_{\sigma(1)}, \ldots, \alpha_{\sigma(i)}\}$ is connected. As one can check, this is always possible by using the assumption that $\{\alpha_1, \ldots, \alpha_k\}$ is connected.
		
		Next, we prove that $m = \overline{\alpha_{\sigma(1)}^{e_1} \alpha_{\sigma(2)}^{e_2} \ldots \alpha_{ \sigma(k)}^{e_k}}$ is a Lyndon element. First note that by the way we constructed $\sigma$, it follows that $\std(m) = \alpha_{\sigma(1)}^{e_1} w$ for some word $w$ in the letters $\alpha_{\sigma(2)}, \ldots , \alpha_{\sigma(k)}$. Since $\alpha_{\sigma(1)} < \alpha_{\sigma(i)}$ for all $2 \leq i \leq k$, it follows $\std(m)$ is minimal in its conjugacy class. To see $\std(m)$ is primitive, we use the assumption that $k \geq 2$ and thus that $w$ is not the empty word. This proves that $\std(m)$ is a Lyndon word and by consequence that $m$ is a Lyndon element.
	\end{proof}
	
	For a function $f:X \to \Z^{\geq 0}$ we let $\supp(f)$ denote the \textit{support of $f$} which is defined by
	\[\supp(f) = \{ x \in X \mid f(x) \neq 0 \}.\]
	
	\begin{lemma}
		\label{lem:nonExistenceLyndonWord}
		Let $\G = (S, E)$ be a graph with an order on $S$ and $w \in W(S)$. If $\supp(e_w)$ is disconnected, then the element $\overline{w}$ is not a Lyndon element.
	\end{lemma}
	\begin{proof}
		Since $\supp(e_w)$ spans a disconnected subgraph of $\G$, we can find a non-trivial partition $\supp(e_w) = S_1 \sqcup S_2$ such that there is no vertex in $S_1$ which is adjacent to a vertex in $S_2$. It follows that there exist non-empty words $u_1 \in W(S_1), \, u_2 \in W(S_2)$ such that $\overline{w} = \overline{u_1 u_2} = \overline{u_2 u_1}$. Since $S_1$ is disjoint from $S_2$, the words $u_1 u_2$ and $u_2 u_1$ do not have the same initial letter. It follows from \cite[Proposition 5.11]{wade15-1} that $\overline{w}$ is not a Lyndon element.
	\end{proof}

	\begin{prop}
		\label{prop:weightsOfLyndonEl}
		Let $\G = (S, E)$ be a graph with an order on $S$. The set of weights of all Lyndon elements is given by
		\[ \{ e_\alpha \mid \alpha \in S \} \cup \left\{ e:S \to \Z^{\geq 0} \; \middle| \; 2 \leq |\supp(e)|, \; \supp(e) \text{ is connected} \right\}. \]
	\end{prop}
	\begin{proof}
		This follows from combining Lemma \ref{lem:existenceLyndonWord} and Lemma \ref{lem:nonExistenceLyndonWord}. Note that the word corresponding to the weight $k e_\alpha$ with $k \geq 2$ is never a Lyndon element, as it is not primitive.
	\end{proof}
	
	\noindent The \textit{bracket words of length $k$} in $S$ are defined inductively as follows:
	\begin{enumerate}
		\item the bracket words of length one are the elements of $S$,
		\item the bracket words of length $k > 1$ are the expressions $[b_1, b_2]$ where $b_1$, $b_2$ are bracket words of lengths $k_1, k_2$ respectively with $k_1 < k$, $k_2 < k$ and $k = k_1 + k_2$.
	\end{enumerate}
	Analogous as for words, we define the \textit{weight $e_b$ of a bracket word $b$} by the function $e_b:S \to \Z^{\geq 0}$ which assigns to each vertex $\alpha$ the number of times it occurs in the bracket word $b$.

	In \cite{wade15-1} a bracketing procedure is described which assigns to every Lyndon element a bracket word with the same weight. By extending this assignment linearly and evaluating the bracket words in $\g_{\G}^K$, we get a map between vector spaces $\phi:K[\LE(M)] \to \g_{\G}^K$. The following was (re)proven in \cite[Corollary 5.24.]{wade15-1}.
	
	\begin{theorem}
		The linear map $\phi:K[\LE(M)] \to \g_{\G}^K$ is bijective.
	\end{theorem}
	
	Note that $\g_{\G}^K$ has a vector space direct sum decomposition
	\[ \g_{\G}^K = \bigoplus_{i = 1}^\infty V^i \quad \quad \quad \text{where} \quad V^1 = V = \spn_K(S) \quad \text{and} \quad V^{i+1} = [V, V^i].\]
	It is clear that under the map $\phi:K[\LE(M)] \to \g_{\G}^K$, Lyndon elements of length $i$ are mapped into $V^i$. Since for any positive integer $k$, we have that $\gamma_{k}(\g_{\G}^K) = \bigoplus_{i = k}^\infty V^i$, it follows that if we write $\LE_c(M)$ for the Lyndon elements of length at most $c$, we have an induced bijection $\phi:K[\LE_c(M)] \to \n_{\G,c}^K$. Equivalently, $\phi(\LE_c(M))$ is a basis for $\n_{\G, c}^K$. Moreover, it is easy to verify that it is a basis of eigenvectors for any vertex diagonal automorphism. In particular, if $m \in \LE_c(S)$ is a Lyndon element and $f$ a vertex diagonal automorphism determined by $\Psi:S \to K$, the eigenvalue of $\phi(m)$ under $f$ is equal to $\prod_{\alpha \in S} \Psi(\alpha)^{e_m(\alpha)}$. Together with Proposition \ref{prop:weightsOfLyndonEl} this immediately proves the main result of this section.
	
	\begin{prop}
		\label{prop:eigenvaluesVertexDiag}
		Let $\n_{\G,c}^K$ be the $c$-step nilpotent Lie algebra associated to the graph $\G = (S, E)$ and let $f_\Psi$ be a vertex-diagonal automorphism of $\n_{\G,c}^K$ determined by the map $\Psi:S \to K$. Then the set of eigenvalues of $f_\Psi$ is given by
		\begin{equation*}
			\Psi(S) \cup
			\left\{ \prod_{\alpha \in S} \Psi(\alpha)^{e(\alpha)} \;\middle|\; 
			\arraycolsep=1.4pt\def\arraystretch{1.3} \begin{array}{c}
				e:S \to \Z^{\geq 0}, \; \supp(e) \text{ is connected}\\
				|\supp(e)| \geq 2, \; \sum_{\alpha\in S} e(\alpha) \leq c 
			\end{array} \right\}.
		\end{equation*}
	\end{prop}
		
	\subsection{Reduction to algebraic units}
	\label{sec:reductionToAlgInt}
	
	As explained under Theorem \ref{thm:injectiveVersionClassificationRationalForms}, we work here with the algebraic closure $\overline{\Q}$ of $\Q$. Let $\G = (S, E)$ be a simple undirected graph and write $\overline{\Q}^S$ for the set of maps from $S$ to $\overline{\Q}$. Note that $\prod_{\lambda \in \Lambda} \Perm(\lambda)$ has a right action on $\overline{\Q}^S$ by precomposition, namely $ \Psi \cdot \theta := \Psi \circ \theta$ for all $\theta \in \prod_{\lambda \in \Lambda} \Perm(\lambda)$ and $\Psi \in \overline{\Q}^S$. Let us write $\mathcal{H}_\G^{\overline{\Q}}$ for the orbit space, so
	\[ \mathcal{H}_\G^{\overline{\Q}} :=  {\overline{\Q}}^S \bigg \slash \prod_{\lambda \in \Lambda} \Perm(\lambda).\]
	
	Note that the groups $\Gal({\overline{\Q}}/\Q)$ and $\Aut(\overline{\G})$ have a well-defined left, respectively right action on $\mathcal{H}_{\G}^{\overline{\Q}}$ by
	\begin{align*}
		\sigma \cdot [\Psi] := [\sigma \circ \Psi], \quad \quad \quad [\Psi]\cdot\varphi := \left[\Psi \circ r(\varphi)\right] 
	\end{align*}
	for all $\sigma \in \Gal({\overline{\Q}}/\Q)$, $\varphi \in \Aut(\overline{\G})$ and $\Psi \in {\overline{\Q}}^S$. Recall that $r$ denotes the morphism $r:\Aut(\overline{\G}) \to \Aut(\G)$ from section \ref{sec:groupOfGradedAutos}, obtained after ordering the vertices in each coherent component. To see that the action by $\Aut(\overline{\G})$ is well defined, one uses the fact that $\prod_{\lambda \in \Lambda} \Perm(\lambda)$ is a normal subgroup of $\Aut(\G)$.
	
	\begin{theorem}
		\label{thm:AnosovEigenvalueCond}
		Let $\G = (S, E)$ be a simple undirected graph and $\rho:\Gal({\Qbar}/\Q) \to \Aut(\overline{\G})$ a continuous morphism. The rational form $\n^\Q_{\rho,c}$ of $\n^\C_{\G,c}$ is Anosov if and only if there exists a map $\Psi:S \to \Qbar$ such that 
		\begin{enumerate}[label = (\roman*)]
			\item \label{item:thmeigenvalcond1} for any vertex $\alpha \in S$, $\Psi(\alpha)$ is an algebraic unit,
			\item \label{item:thmeigenvalcond2} for any $1 \leq n \leq c$ and any (not necessarily distinct) vertices $\alpha_1, \ldots, \alpha_n \in S$ such that $\{\alpha_1, \ldots, \alpha_n\}$ spans a connected subgraph of $\G$, we have $|\Psi(\alpha_1) \cdot \ldots \cdot \Psi(\alpha_n)| \neq 1$,
			\item \label{item:thmeigenvalcond3} and for any $\sigma \in \Gal({\Qbar}/\Q)$ it holds that $\sigma \cdot [\Psi]  = [\Psi] \cdot \rho_\sigma$ in $\mathcal{H}_\G^{\Qbar}$.
		\end{enumerate}
	\end{theorem}
	Note that the conditions in this theorem do not depend on the choice of representative $\Psi$ in its equivalence class $[\Psi] \in \mathcal{H}^{\Qbar}_\G$.
	\begin{proof}
		
		First assume that $\n_{\rho, c}^\Q \subset \n_{\G, c}^{\Qbar}$ is Anosov with Anosov automorphism $f:\n_{\rho,c}^\Q \to \n_{\rho,c}^\Q$. Recall that $p, \pi$ and $i$ are the morphisms as defined in section \ref{sec:groupOfGradedAutos} by (\ref{eq:morphismp}), (\ref{eq:morphismpi}) and (\ref{eq:morphismi}), respectively, and that
		\[G = p\left(\Aut\left(\n_{\G, c}^{\Qbar}\right)\right) = M \cdot \left(\prod_{\lambda \in \Lambda } \GL(V_\lambda)\right) \cdot \overline{P}(\Aut(\overline{\G})).\]
		Note that since $\n_{\rho, c}^\Q$ is a rational form of $\n^{\Qbar}_{\G, c}$, we can naturally extend $f$ to an automorphism of $\n^{\Qbar}_{\G, c}$. The semisimple part of $f$ is again an Anosov automorphism, as $\Aut(\n_{\rho, c}^\Q)$ is a linear algebraic group. Moreover, any positive power $f^k$ for $k > 0$ of $f$ is an Anosov automorphism as well, so without loss of generality we can assume that $f$ is semi-simple and lies in the Zariski-connected component of $\Aut\left(\n_{\G,c}^{\Qbar}\right)$. Hence, under these assumptions, $f$ lies in some maximal torus of $\Aut\left(\n^{\Qbar}_{\G,c}\right)$. We know as well that the subgroup of vertex diagonal automorphisms $D_{\G,c}$ is a maximal torus of $\Aut\left(\n^{\Qbar}_{\G,c}\right)$. Since $\Aut\left(\n^{\Qbar}_{\G,c}\right)$ is a linear algebraic group over an algebraically closed field, all its maximal tori are conjugate and thus there exists an $h \in \Aut\left(\n^{\Qbar}_{\G,c}\right)$ and an $\tilde{f} \in D_{\G,c}$ such that $h \tilde{f} h^{-1} = f$. Moreover, since $i\left(\Aut\left(\overline{\G}\right)\right)$ normalizes $D_{\G,c}$, we can assume that $\pi(h) = 1$.
		
		Let us define $\Psi:S \to \Qbar$ by assigning to a vertex $\alpha \in S$ its corresponding eigenvalue under $\tilde{f}$. Since $\tilde{f}$ is an Anosov automorphism, it follows that $\Psi(\alpha)$ is a hyperbolic algebraic unit for all $\alpha \in S$. As a consequence we also have $|\Psi(\alpha)^k| \neq 0$ for any $\alpha \in S$ and $k \in \Z^{> 0}$. By Proposition \ref{prop:eigenvaluesVertexDiag}, we know that for any vertices $\alpha_1,\ldots,\alpha_n \in S$ with $|\{\alpha_1,\ldots,\alpha_n\}| \geq 2$ and such that $\{\alpha_1,\ldots,\alpha_n\}$ spans a connected subgraph of $\G$, the product $\Psi(\alpha_1)\cdot\ldots \cdot \Psi(\alpha_n)$ is an eigenvalue of $\tilde{f}$ and thus $|\Psi(\alpha_1)\cdot\ldots \cdot \Psi(\alpha_n)| \neq 1$. All together this proves $\Psi$ satisfies conditions \ref{item:thmeigenvalcond1} and \ref{item:thmeigenvalcond2}.
		
		Let us now show $\Psi$ also satisfies condition \ref{item:thmeigenvalcond3}. Since $f\left(\n_{\rho, c}^\Q\right) = \n_{\rho, c}^\Q$, it follows from Lemma \ref{lem:autosOfRatForm} that $f  \, i(\rho_\sigma) = i(\rho_\sigma)  \asi{f}$ for all $\sigma \in  \Gal(\Qbar/\Q)$. Substituting $h\tilde{f}h^{-1}$ for $f$, we find the equality 
		\begin{equation}
			\label{eq:proofThmEigenvalues1}
			\tilde{f} a_\sigma = a_\sigma \asi{\tilde{f}} \quad \quad \text{where} \quad  a_\sigma := h^{-1} i(\rho_\sigma) \asi{h}.
		\end{equation}
		Note that $\pi(a_\sigma) = \pi(h^{-1}) \pi(i(\rho_\sigma)) \pi(h) = \pi(i( \rho_\sigma)) = \rho_\sigma$, and hence $p(a_\sigma) = m_\sigma A_\sigma \overline{P}(\rho_\sigma)$ for some unique $m_\sigma \in M, A_\sigma \in \prod_{\lambda \in \Lambda} \GL(V_\lambda)$. Applying the morphism $p:\Aut\left(\n_{\G, c}^{\Qbar}\right) \to G$ to equation (\ref{eq:proofThmEigenvalues1}), we get that
		\[ p(\tilde{f}) m_\sigma A_\sigma \overline{P}(\rho_\sigma) =  m_\sigma A_\sigma \overline{P}(\rho_\sigma ) \asi{p(\tilde{f})}.\]
		Note that $p(\tilde{f})$ and $\asi{p(\tilde{f})}$ are elements in $\prod_{\lambda \in \Lambda} \GL(V_\lambda)$, since $\tilde{f}$ is diagonal on $S$. Rearranging the equation above to
		\[ \Big( p(\tilde{f}) m_\sigma p(\tilde{f})^{-1} \Big) \Big( p(\tilde{f}) A_\sigma \Big) \Big(\overline{P}( \rho_\sigma ) \Big) =  \Big(m_\sigma \Big) \Big( A_\sigma \overline{P}(\rho_\sigma) \asi{p(\tilde{f})} \overline{P}( \rho_\sigma)^{-1} \Big) \Big( \overline{P}(\rho_\sigma) \Big) \]
		we can find the equality on $\prod_{\lambda \in \Lambda} \GL(V_\lambda)$ to be
		\[ A_\sigma^{-1} \, p(\tilde{f}) \,  A_\sigma = \overline{P}( \rho_\sigma )\, \asi{p(\tilde{f})} \, \overline{P}( \rho_\sigma)^{-1} \]
		Since both sides are elements of $\prod_{\lambda \in \Lambda} \GL(V_\lambda)$, we can look at their projection onto $\GL(V_\lambda)$ for any $\lambda \in \Lambda$. We then find that $\tilde{f}|_{V_\lambda}$ and  $\displaystyle \asi{\left(\tilde{f}|_{V_{ \rho_\sigma^{-1} (\lambda)}}\right)}$ are conjugate and thus their eigenvalues, counted with multiplicities, coincide. This shows exactly that $[\Psi] = [\sigma \circ \Psi \circ r(\rho_\sigma^{-1})]$ for any $\sigma \in \Gal(\Qbar/\Q)$. This proves that $\Psi$ satisfies all the required properties.
		
		Conversely, assume that a map $\Psi: S \to \Qbar$ satisfying conditions \ref{item:thmeigenvalcond1}, \ref{item:thmeigenvalcond2} and \ref{item:thmeigenvalcond3} exists. Let $m$ be the number of orbits for the $\rho$-action of $\Gal(\Qbar/\Q)$ on $\Lambda$ and choose coherent components $\lambda_1, \ldots, \lambda_m \in \Lambda$ such that $\Lambda = \bigsqcup_{j = 1}^m \Orb_{\rho}(\lambda_i)$. For any $j \in \E{m}$, define the polynomial $g_j(X) \in \Qbar[X]$ by
		\[g_j(X) = \prod_{\alpha \in \lambda_j} (X - \Psi(\alpha)).\]
		Let us fix an action of $\Gal(\Qbar/\Q)$ on $\Qbar[X]$ by acting on the coefficients of the polynomials. As one can check this is an action by ring automorphisms. Take an arbitrary $\sigma \in \Stab_\rho(\lambda_j)$. By the assumption, we have that $[\sigma \circ \Psi] = [\Psi \circ r(\rho_\sigma)]$. By consequence, there exists a $\theta \in \prod_{\lambda \in \Lambda} \Perm(\lambda)$ such that $\sigma \circ \Psi = \Psi \circ r(\rho_\sigma) \circ \theta$. We now have that
		\begin{align}
			\label{eq:proofCondEigen2}
			\asi{g_j(X)} &= \prod_{\alpha \in \lambda_j} \asi{\left(X - \Psi(\alpha)\right)}\nonumber\\
			&= \prod_{\alpha \in \lambda_j} \left(X - (\sigma \circ \Psi)(\alpha)\right)\nonumber\\
			&= \prod_{\alpha \in \lambda_j} \left(X - (\Psi \circ r(\rho_\sigma) \circ \theta)(\alpha)\right)\nonumber\\
			&= \prod_{\alpha \in \lambda_j} \left(X - (\Psi \circ r(\rho_\sigma))(\alpha)\right)\nonumber\\
			&= \prod_{\alpha \in \rho_\sigma(\lambda_i)} \left(X - \Psi(\alpha)\right)\nonumber\\
			&= \prod_{\alpha \in \lambda_j} \left(X - \Psi(\alpha)\right) = g_j(X).
		\end{align}
		So the coefficients of $g_j(X)$ lie in the (finite) field extension ${\Qbar}^{\Stab_\rho(\lambda_j)}/\Q$. 
		
		Next, for any $j \in \{1, \ldots, m\}$, let $B_j \in \GL(V_{\lambda_j})$ be the linear map given by the companion matrix of $g_j(X)$ in a basis of vertices of $V_{\lambda_j}$, where the order of the basis does not matter. Clearly we have that $\asi{B_j} = B_j$ for any $\sigma \in \Stab_{\rho}(\lambda_j)$. Now define the linear map $A:V \to V$ by setting for any $\mu \in \Orb_\rho(\lambda_j)$ and $v \in V_\mu$:
		\[ A(v) = i(\rho_\sigma) \, \asi{B_j} \, i( \rho_\sigma)^{-1} v \]
		where $\sigma \in \Gal({\Qbar}/\Q)$ is chosen such that $\sigma(\lambda_j) = \mu$. Let us first show this is well-defined and independent of the choice of $\sigma \in \Gal({\Qbar}/\Q)$. Say $\nu \in \Gal({\Qbar}/\Q)$ is another element which also satisfies $\nu(\lambda_j) = \mu$, then $ \sigma^{-1} \nu \in \text{stab}_\rho(\lambda_j)$ and we get
		\begin{align*}
			i(\rho_\nu) \, \prescript{\nu}{}{B_j} \, i(\rho_\nu)^{-1} v &= i(\rho_\sigma) i(\rho_{ \sigma^{-1}\nu}) \, \asi{\left( \prescript{\sigma^{-1} \nu}{}{B_j} \right)} i(\rho_{ \sigma^{-1}\nu})^{-1} i(\rho_\sigma)^{-1} v\\
			&= i(\rho_\sigma)  \, \asi{B_j} i(\rho_\sigma)^{-1} v.
		\end{align*}
		since $i(\rho_{\sigma^{-1}\nu})|_{V_{\lambda_j}} = \Id_{V_{\lambda_j}}$. We thus have a well-defined linear map $A:V \to V$ and as one can check $A \in \prod_{\lambda \in \Lambda} \GL(V_{\lambda}) \subset p(T_{\G,c})$. This gives a unique automorphism $f \in T_{\G,c} \subset \Aut(\n_\G^{\Qbar})$ with $p(f) = A$. 
		
		We claim that $f$ induces an Anosov automorphism on $\n^{\Q}_{\rho, c}$. To check that $f(\n^\Q_{\rho, c}) = \n^\Q_{\rho, c}$, we need to check that $i(\rho_\sigma)^{-1} \,f \, i(\rho_\sigma) =   \asi{f}$ for any $\sigma \in \Gal({\Qbar}/\Q)$. Since both $f$ and $i(\rho_\sigma)$ lie in $T_{\G,c} \subset \Aut(\n_{\G, \rho}^{\Qbar})$, it suffices to check this on $V$. Take any $j \in \{1,\ldots, m\}$, $\mu \in \Orb_{\rho}(\lambda_j)$ and $v \in V_\mu$. Let $\nu \in \Gal({\Qbar}/\Q)$ be an element such that $\rho_\nu(\lambda_j) = \rho_\sigma(\mu)$ or equivalently $\rho_{\sigma^{-1} \nu}(\lambda_j) = \mu$. Then we have
		\begin{align*}
			i(\rho_\sigma)^{-1} \,f \, i(\rho_\sigma) v &= i(\rho_\sigma)^{-1} i(\rho_\nu) \, \ata{B_j} \, i(\rho_\nu)^{-1} i(\rho_\sigma) v\\
			&= i(\rho_{\sigma^{-1} \nu}) \, \asi{\left( \prescript{\sigma^{-1} \nu}{}{B_j} \right)} \, i(\rho_{\nu^{-1} \sigma}) v\\
			&= \asi{\left( i(\rho_{\sigma^{-1} \nu}) \, \prescript{\sigma^{-1} \nu}{}{B_j} \, i(\rho_{\sigma^{-1} \nu})^{-1} \right)} v\\
			&= \asi{f} v.
		\end{align*}
		The set of eigenvalues of $f$ on $V$ is equal to the image of $\Psi$ which are algebraic units by assumption. The other eigenvalues of $f$ are products of the eigenvalues on $V$ and are thus algebraic units as well. This shows that $f$ is integer-like, since its characteristic polynomial has coefficients in $\Q$ from the fact that $f(\n^{\Q}_{\rho, c}) = \n^{\Q}_{\rho, c}$ as proven above. To see that $f$ is hyperbolic, note that by the way we constructed $f$, it follows that $p(f)$ is conjugated to $p(f_{\Psi})$ by an element of $\prod_{\lambda \in \Lambda} \GL(V_\lambda)$. By consequence $f$ is conjugated to the vertex diagonal automorphism $f_{\Psi}$ in $\Aut(\n_{\G, c}^{\overline{\Q}})$. We thus only need to check that $f_{\Psi}$ is hyperbolic. This follows straightforward from the assumption on $\Psi$ and Proposition \ref{prop:eigenvaluesVertexDiag}.
	\end{proof}
	
	\section{Proof of Theorem \ref{thm:AnosovOrbitCondInjective}}
	\label{sec:Anosov}

	In this section, we give a condition for a rational form $\n^\Q_{\rho, c}$ to be Anosov which is easier to check than the condition in Theorem \ref{thm:AnosovEigenvalueCond} and which solely depends on how the orbits of the action on $\overline{\G}$ induced by $\rho$ look like and on which coherent components are fixed under the action of the complex conjugation automorphism. Before we prove this characterization, we first prove three lemmas. At the end of the section we list some corollaries, one of which is a correction to a result of \cite{main06-1}.
	
	A number field is said to be \textit{totally imaginary} if it has no real embeddings (see section \ref{sec:numberFieldsAndAlgUnitsWithHypProp}). In equation (\ref{eq:definitionz_rho}) of the introduction we defined a function $z_\rho$ for any morphism $\rho:\Gal(L/\Q) \to \Aut(\G)$. The following lemma motivates this definition. Note that if $\overline{\rho}:\Gal(\Qbar/\Q) \to \Aut(\overline{\G}):\sigma \mapsto \rho(\sigma|_L)$ denotes the extended morphism, then the associated function $z_{\overline{\rho}}$ is equal to $z_\rho$. As mentioned in the introduction, $z_\rho$ is constant on $\rho$-orbits. As usual, we let $\tau$ denote the complex conjugation automorphism and for any subgroup $H \leq \Gal(\Qbar/\Q)$ we write $\Qbar^H$ for the field consisting of all elements of $\Qbar$ fixed by all automorphisms of $H$.
	
	\begin{lemma}
		\label{lem:totallyImaginaryStabField}
		Let $\G = (S, E)$ be a graph with set of coherent components $\Lambda$ and $\rho:\Gal(\Qbar/\Q) \to \Aut(\overline{\G})$ a continuous morphism. For any $\lambda \in \Lambda$, the field $\Qbar^{H}$ with $H = \Stab_\rho(\lambda)$ is totally imaginary if and only if $z_\rho(\lambda) = 1/2$.
	\end{lemma}

	\begin{proof}
		Note that $\Qbar^H$ is totally imaginary if and only if for any $\sigma \in \Gal(\Qbar/\Q)$ we have
		\[\tau \notin \Gal\left(\Qbar\middle/\sigma\left(\Qbar^H\right)\right).\]
		We have the series of equivalences
		\begin{align*}
			\tau \in \Gal\left(\Qbar\middle/\sigma\left({\Qbar}^H \right)\right) &\Leftrightarrow \tau \in \Gal\left(\Qbar\middle/{\Qbar}^{\sigma H \sigma^{-1}}\right)\\
			&\Leftrightarrow \tau \in \sigma H \sigma^{-1}\\
			&\Leftrightarrow \tau \in \Stab_\rho(\rho_\sigma(\lambda))\\
			&\Leftrightarrow \rho_{\tau \sigma}(\lambda) = \rho_\sigma(\lambda).
		\end{align*}
		Thus, $\Qbar^H$ is totally imaginary if and only if for any $\sigma \in \Gal(\Qbar/\Q)$ we have $\rho_{\tau\sigma}(\lambda) \neq \rho_\sigma(\lambda)$, which is by definition equivalent with $z_\rho(\lambda) = 1/2$.
	\end{proof}
	
	The next lemma is a result on the existence of connected subsets of coherent components satisfying a property with respect to a given involution of the quotient graph.
	
	\begin{lemma}
		\label{lem:existenceSetB}
		Let $\G = (S,E)$ be a graph, $\Lambda$ its set of coherent components and $\iota \in \Aut(\overline{\G})$ an involution, i.e. $\iota^2 = \Id$. Take any connected subset $A \subset \Lambda$ such that for all $\lambda \in \Lambda$ with $\iota(\lambda) \neq \lambda$ it holds that $A \neq \{\lambda, \iota(\lambda)\}$. Then there exists a connected subset $B \subset \Lambda$ such that $A \cup \iota(A) = B \cup \iota(B)$ and for any $\lambda \in B$ with $\lambda \neq \iota(\lambda)$ we have $\iota(\lambda) \notin B$.
	\end{lemma}

	Before proving it, we first illustrate it with $2$ different examples. 
	\begin{example}
		\label{ex:completeBiparteGraph}
		The above lemma requires the set $A \subset \Lambda$ to satisfy $A \neq \{\lambda, \iota(\lambda)\}$ for any $\lambda \in \Lambda$ with $\lambda \neq \iota(\lambda)$. This condition is indeed necessary by considering the following example. Take any positive integer $n > 1$ and let $S = \{\alpha_1,\ldots, \alpha_n, \beta_1, \ldots, \beta_n \}$. Define $E = \{\{\alpha_k, \beta_l\} \mid 1 \leq k, l \leq n\}$. The resulting graph $\G = (S, E)$ is the \textit{complete biparte graph on $n + n$ vertices} and is drawn below for $n = 3$.
		
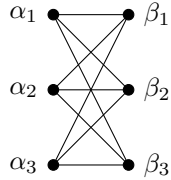
\begin{figure}[H]
			\centering
			\begin{tikzpicture}
				\filldraw (-0.5,0) circle (2pt) node[anchor = east] {$\alpha_2\:$};
				\filldraw (-0.5,1) circle (2pt) node[anchor = east] {$\alpha_1\:$};
				\filldraw (-0.5,-1) circle (2pt) node[anchor = east] {$\alpha_3\:$};
				\filldraw (0.5,0) circle (2pt) node[anchor = west] {$\:\beta_2$};
				\filldraw (0.5,1) circle (2pt) node[anchor = west] {$\:\beta_1$};
				\filldraw (0.5,-1) circle (2pt) node[anchor = west] {$\:\beta_3$};
				\draw (-0.5,0) -- (0.5,0);
				\draw (-0.5,0) -- (0.5,1);
				\draw (-0.5,0) -- (0.5,-1);
				\draw (-0.5,1) -- (0.5,0);
				\draw (-0.5,1) -- (0.5,1);
				\draw (-0.5,1) -- (0.5,-1);
				\draw (-0.5,-1) -- (0.5,0);
				\draw (-0.5,-1) -- (0.5,1);
				\draw (-0.5,-1) -- (0.5,-1);
			\end{tikzpicture}
		\end{figure}
		The coherent components are given by $\Lambda = \{ \lambda_1 = \{\alpha_1, \alpha_2, \alpha_3\}, \lambda_2 = \{\beta_1, \beta_2, \beta_3\} \}$. Note that $A := \{\lambda_1, \lambda_2\}$ is connected since $\lambda_1 \cup \lambda_2 = S$ is a connected set of vertices. Let $\iota \in \Aut(\overline{\G})$ denote the involution defined by $\iota(\lambda_1) = \lambda_2$. It is clear that the above lemma can not be valid for this choice of set $A$ since $\lambda_1$ and $\lambda_2$ are each not connected subsets in $\G$ and thus $\{\lambda_1\}$ and $\{\lambda_2\}$ are each not a connected subset of coherent components.
	\end{example}
	\begin{example}
		To illustrate the conclusion of Lemma \ref{lem:existenceSetB}, consider the graph $\G = (S, E)$ with $S = \{\alpha_1, \ldots, \alpha_7\}$ and $E = \{\{\alpha_1, \alpha_2\}, \{\alpha_2, \alpha_3\}, \{\alpha_1, \alpha_4\}, \{\alpha_4, \alpha_6\},  \{\alpha_3, \alpha_5\}, \{\alpha_5, \alpha_7\}, \{\alpha_4, \alpha_5\}\}$. The coherent components are simply all the singletons $\Lambda = \{\lambda_i = \{\alpha_i\} \mid 1 \leq i \leq 7\}$. Define the subset of coherent components $A = \{\lambda_2, \lambda_3, \lambda_5, \lambda_4, \lambda_6\} \subset \Lambda$, which is easily seen to be connected. Let $\iota \in \Aut(\overline{\G})$ be the involution defined by $\iota(\lambda_{2}) = \lambda_{2}$, $\iota(\lambda_1) = \lambda_3$, $\iota(\lambda_4) = \lambda_5$ and $\iota(\lambda_6) = \lambda_7$. The graph and its quotient graph are drawn below. The subset $A$ is drawn in red.
		
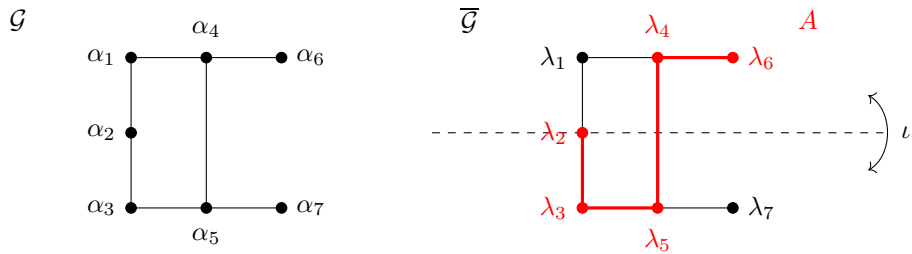
\begin{figure}[H]
			\centering
			\begin{tikzpicture}
				\begin{scope}[shift={(-3, 0)}]
					\draw (-1, 1) -- (-1, 0);
					\draw (-1, 0) -- (-1, -1);
					\draw (-1, 1) -- (0, 1);
					\draw (0, 1) -- (1, 1);
					\draw (-1, -1) -- (0, -1);
					\draw (0, -1) -- (1, -1);
					\draw (0, 1) -- (0, -1);
					\filldraw (-1,0) circle (2pt) node[anchor = east] {$\alpha_2\:$};
					\filldraw (-1,1) circle (2pt) node[anchor = east] {$\alpha_1\:$};
					\filldraw (-1,-1) circle (2pt) node[anchor = east] {$\alpha_3\:$};					
					\filldraw (0,1) circle (2pt) node[above = 0.15] {$\alpha_4$};
					\filldraw (0,-1) circle (2pt) node[below = 0.15] {$\alpha_5$};					
					\filldraw (1,1) circle (2pt) node[anchor = west] {$\:\alpha_6$};
					\filldraw (1,-1) circle (2pt) node[anchor = west] {$\:\alpha_7$};
					\node at (-2.5,1.5) {$\G$};
				\end{scope}
				\begin{scope}[shift={(3, 0)}]
					\draw[dashed] (-3, 0) -- (3, 0);
					\draw (-1, 1) -- (-1, 0);
					\draw[very thick, red] (-1, 0) -- (-1, -1);
					\draw (-1, 1) -- (0, 1);
					\draw[very thick,color=red] (0, 1) -- (1, 1);
					\draw[very thick,color=red] (-1, -1) -- (0, -1);
					\draw (0, -1) -- (1, -1);
					\draw[very thick,color=red] (0, 1) -- (0, -1);
					\filldraw[color=red] (-1,0) circle (2pt) node[anchor = east] {$\lambda_2\:$};
					\filldraw (-1,1) circle (2pt) node[anchor = east] {$\lambda_1\:$};
					\filldraw[color=red] (-1,-1) circle (2pt) node[anchor = east] {$\lambda_3\:$};					
					\filldraw[color=red] (0,1) circle (2pt) node[above = 0.15] {$\lambda_4$};
					\filldraw[color=red] (0,-1) circle (2pt) node[below = 0.15] {$\lambda_5$};					
					\filldraw[color=red] (1,1) circle (2pt) node[anchor = west] {$\:\lambda_6$};
					\filldraw (1,-1) circle (2pt) node[anchor = west] {$\:\lambda_7$};
					\draw [<->] (2.8, -0.5) to [out=30,in=-30] (2.8, 0.5);
					\node at (3.3,0) {$\iota$};
					\node at (-2.5,1.5) {$\overline{\G}$};
					\node[color=red] at (2,1.5) {$A$};
				\end{scope}
			\end{tikzpicture}
		\end{figure}
		\noindent Lemma \ref{lem:existenceSetB} gives a set $B$ such that $B \cup \iota(B) = A \cup \iota(A)$ and for any $\lambda \in B$ with $\lambda \neq \iota(\lambda)$ it holds that $\iota(\lambda) \notin B$. In this example, such a set $B$ can be given by either $B = \{ \lambda_2, \lambda_3, \lambda_5, \lambda_7 \}$ or $B = \{\lambda_2, \lambda_1, \lambda_4, \lambda_6\}$. These sets are drawn below in blue.
		
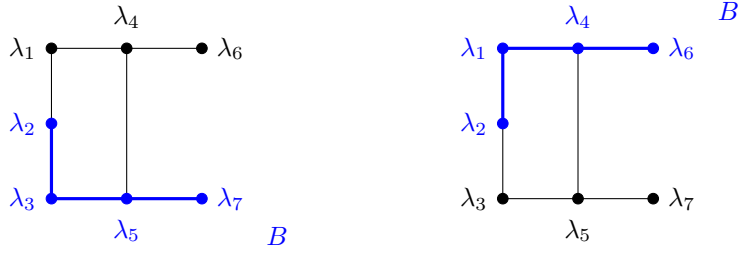
\begin{figure}[H]
			\centering
			\begin{tikzpicture}
				\begin{scope}[shift={(-3, 0)}]
					\draw (-1, 1) -- (-1, 0);
					\draw[very thick,color=blue] (-1, 0) -- (-1, -1);
					\draw (-1, 1) -- (0, 1);
					\draw (0, 1) -- (1, 1);
					\draw[very thick,color=blue] (-1, -1) -- (0, -1);
					\draw[very thick,color=blue] (0, -1) -- (1, -1);
					\draw (0, 1) -- (0, -1);
					\filldraw[color=blue] (-1,0) circle (2pt) node[anchor = east] {$\lambda_2\:$};
					\filldraw (-1,1) circle (2pt) node[anchor = east] {$\lambda_1\:$};
					\filldraw[color=blue] (-1,-1) circle (2pt) node[anchor = east] {$\lambda_3\:$};					
					\filldraw (0,1) circle (2pt) node[above = 0.15] {$\lambda_4$};
					\filldraw[color=blue] (0,-1) circle (2pt) node[below = 0.15] {$\lambda_5$};
					\filldraw (1,1) circle (2pt) node[anchor = west] {$\:\lambda_6$};
					\filldraw[color=blue] (1,-1) circle (2pt) node[anchor = west] {$\:\lambda_7$};
					\node[color=blue] at (2,-1.5) {$B$};
				\end{scope}
				\begin{scope}[shift={(3, 0)}]
					\draw[very thick,color=blue] (-1, 1) -- (-1, 0);
					\draw (-1, 0) -- (-1, -1);
					\draw[very thick,color=blue] (-1, 1) -- (0, 1);
					\draw[very thick,color=blue] (0, 1) -- (1, 1);
					\draw (-1, -1) -- (0, -1);
					\draw (0, -1) -- (1, -1);
					\draw (0, 1) -- (0, -1);
					\filldraw[color=blue] (-1,0) circle (2pt) node[anchor = east] {$\lambda_2\:$};
					\filldraw[color=blue] (-1,1) circle (2pt) node[anchor = east] {$\lambda_1\:$};
					\filldraw (-1,-1) circle (2pt) node[anchor = east] {$\lambda_3\:$};
					\filldraw[color=blue] (0,1) circle (2pt) node[above = 0.15] {$\lambda_4$};
					\filldraw (0,-1) circle (2pt) node[below = 0.15] {$\lambda_5$};
					\filldraw[color=blue] (1,1) circle (2pt) node[anchor = west] {$\:\lambda_6$};
					\filldraw (1,-1) circle (2pt) node[anchor = west] {$\:\lambda_7$};
					\node[color=blue] at (2, 1.5) {$B$};
				\end{scope}
			\end{tikzpicture}
		\end{figure}
	\end{example}
	
	\begin{proof}[Proof of Lemma \ref{lem:existenceSetB}:]
		We will first prove the cases where $A$ counts $1$, $2$ or $3$ elements and then proceed by induction on $|A|$. If $|A| = 1$, we can just take $B = A$. If $|A| = 2$, the assumptions on $A$ imply that $\iota(\lambda) \neq \lambda$ for any $\lambda \in A$. Thus we can again take $B = A$. If $|A| = 3$, we can assume that there is a coherent component $\lambda \in A$ such that $\iota(\lambda) \neq \lambda$ and $\iota(\lambda) \in A$, otherwise we can just take $B = A$. Without loss of generality we can then write $A = \{\lambda_1, \lambda_2, \lambda_3\}$ with $\iota(\lambda_2) = \lambda_3$. Since $A$ is connected, we must have that either $\{\lambda_1, \lambda_2\} \in \overline{E}$, in which case we take $B = \{\lambda_1, \lambda_2\}$ or that $\{\lambda_1, \lambda_3\} \in \overline{E}$, in which case we take $B = \{\lambda_1, \lambda_3\}$.
		
		Next, assume that $|A| > 3$ and that the theorem holds for lower cardinalities of $A$. As a basic property of connected subsets of graphs (which also lifts to connected sets of coherent components containing at least 3 coherent components), there exists a coherent component $\lambda \in A$ such that $A' = A \setminus \{\lambda\}$ is still connected. In addition, there must exist an element $\mu \in A'$ such that $\{\lambda, \mu\} \in \overline{E}$ since $A$ was assumed connected. We can apply the induction hypothesis on $A'$ and get a connected set $B' \subset \Lambda$ such that $A' \cup \iota(A') = B' \cup \iota(B')$ and for any $\lambda \in B'$ with $\lambda \neq \iota(\lambda)$ we have $\iota(\lambda) \notin B'$. The following three cases are to be considered:
		\begin{itemize}
			\item $\iota(\lambda) \in A'$. In this case, we have that $A \cup \iota(A) = A' \cup \iota(A') = B' \cup \iota(B')$. By consequence we find that $B := B'$ satisfies the required properties.
			
			\item $\iota(\lambda) \notin A'$ and $\mu \in B'$. Then it follows from $\{\lambda, \mu\} \in \overline{E}$, that $B := B' \cup \{\lambda\}$ is connected. We also have
			\begin{align*}
				B \cup \iota(B) &= B' \cup \{\lambda\} \cup \iota(B' \cup \{\lambda\})\\
				&= B' \cup \iota(B') \cup \{\lambda\} \cup \{\iota(\lambda)\}\\
				&= A' \cup \iota(A') \cup \{\lambda\} \cup \{\iota(\lambda)\}\\
				&= A' \cup \{\lambda\} \cup \iota(A' \cup \{\lambda\})\\
				&= A \cup \iota(A)
			\end{align*}
			and $\iota(\lambda) \notin B$ for any $\lambda \in B$ with $\lambda \neq \iota(\lambda)$. Thus $B$ satisfies all required properties.
			
			\item $\iota(\lambda) \notin A'$ and $\mu \notin B'$. Then because $\mu \in A' \cup \iota(A') = B' \cup \iota(B')$, we must have $\iota(\mu) \in B'$. Since $\iota \in \Aut(\overline{\G})$ and $\{\lambda, \mu\} \in \overline{E}$, we must have that $\{ \iota(\lambda), \iota(\mu) \} \in \overline{E}$. By consequence we have that $B := B' \cup \{ \iota(\lambda) \}$ is connected. We also have
			\begin{align*}
				B \cup \iota(B) &= B' \cup \{\iota(\lambda)\} \cup \iota(B' \cup \{\iota(\lambda)\})\\
				&= B' \cup \iota(B') \cup \{\lambda\} \cup \{\iota(\lambda)\}\\
				&= A' \cup \iota(A') \cup \{\lambda\} \cup \{\iota(\lambda)\}\\
				&= A' \cup \{\lambda\} \cup \iota(A' \cup \{\lambda\})\\
				&= A \cup \iota(A)
			\end{align*}
			and $\iota(\lambda) \notin B$ for any $\lambda \in B$ with $\lambda \neq \iota(\lambda)$. Thus $B$ satisfies all required properties.
		\end{itemize}
		This concludes the proof.
	\end{proof}
		
	At last, we prove a lemma which helps us find eigenvalues equal to $\pm 1$ in certain vertex diagonal automorphism.
	
	\begin{lemma}
		\label{lem:productEqualsPmOne}
		Let $\G = (S, E)$ be a graph and $\Psi:S \to \Qbar$ a function satisfying conditions \ref{item:thmeigenvalcond1} and \ref{item:thmeigenvalcond3} of Theorem \ref{thm:AnosovEigenvalueCond}. If $A \subset \Lambda$ is a $\rho$-invariant subset and $f:\Lambda \to \Z^{\geq 0}$ a $\rho$-invariant function, then
		\[ \prod_{\lambda \in A} \prod_{\alpha \in \lambda} \Psi(\alpha)^{f(\lambda)} = \pm 1.\]
	\end{lemma}
	\begin{proof}
		Take any $\sigma \in \Gal(\Qbar/\Q)$. Since $\Psi$ satisfies condition \ref{item:thmeigenvalcond3}, there exists a $\theta \in \prod_{\lambda \in \Lambda} \Perm(\lambda)$ such that $\sigma \circ \Psi = \Psi \circ r(\rho_\sigma) \circ \theta$. We thus have
		\begin{align*}
			\sigma\left(\prod_{\lambda \in A} \prod_{\alpha \in \lambda} \Psi(\alpha)^{f(\lambda)}\right) &= \prod_{\lambda \in A} \prod_{\alpha \in \lambda} (\sigma \circ \Psi)(\alpha)^{f(\lambda)}\\
			&=\prod_{\lambda \in A} \prod_{\alpha \in \lambda} (\Psi \circ r(\rho_\sigma) \circ \theta)(\alpha)^{f(\lambda)}\\
			&=\prod_{\lambda \in A} \prod_{\alpha \in \lambda} (\Psi \circ r(\rho_\sigma))(\alpha)^{(f \circ \rho_\sigma)(\lambda)}\\
			&=\prod_{\lambda \in A} \prod_{\alpha \in \rho_\sigma(\lambda)} \Psi(\alpha)^{f(\lambda)}\\
			&= \prod_{\lambda \in A} \prod_{\alpha \in \lambda} \Psi(\alpha)^{f(\lambda)}.
		\end{align*}
		This proves that $\left(\prod_{\lambda \in A} \prod_{\alpha \in \lambda} \Psi(\alpha)^{f(\lambda)}\right) \in \Q$. Since $\Psi$ satisfies condition \ref{item:thmeigenvalcond1}, every factor in this product is an algebraic unit and thus the product itself is an algebraic unit. The only algebraic units in $\Q$ are $1$ and $-1$, which proves the claim.
	\end{proof}
	
	We are now ready to prove the main theorem of the section, which is equivalent to Theorem \ref{thm:AnosovOrbitCondInjective}.

	\begin{theorem}
		\label{thm:AnosovOrbitCond}
		Let $\G = (S, E)$ be a simple undirected graph and $\rho: \Gal(\Qbar/\Q) \to \Aut(\overline{\G})$ a continuous morphism. Then the associated rational form $\n_{\rho, c}^\Q$ of $\n_{\G, c}^\C$ is Anosov if and only if for each non-empty connected set of coherent components $A \subset \Lambda$ such that $A \cup \rho_\tau(A)$ is $\rho$-invariant, it holds that
		\begin{equation}
			\label{eq:inequalityOrbitCondition}
			c < \sum_{\lambda \in A \cup \rho_\tau(A)} z_\rho(\lambda) \cdot |\lambda|.
		\end{equation}
	\end{theorem}
	\begin{proof}
		Assume that $\n_{\rho, c}^\Q$ is Anosov. By Theorem \ref{thm:AnosovEigenvalueCond}, there exists a map $\Psi:S \to \Qbar$ which satisfies conditions \ref{item:thmeigenvalcond1}, \ref{item:thmeigenvalcond2} and \ref{item:thmeigenvalcond3} from that theorem. Take any non-empty connected subset $A \subset S$ such that $A \cup \rho_\tau(A)$ is $\rho$-invariant. We will prove that the inequality (\ref{eq:inequalityOrbitCondition}) holds.
		
		First assume that $A$ is not of the form $\{\lambda , \rho_\tau(\lambda)\}$ for some $\lambda \in \Lambda$ with $\rho_\tau(\lambda) \neq \lambda$. By Lemma \ref{lem:existenceSetB}, there exists a connected set $B \subset \Lambda$ such that $A \cup \rho_\tau(A) = B \cup \rho_\tau(B)$. Define the function
		\[ g:B \to \{1, 2\}: \lambda \mapsto \begin{cases}
			1 &\text{ if } z_\rho(\lambda) = 1/2\\
			2 &\text{ if } z_\rho(\lambda) = 1 \text{ and } \rho_\tau(\lambda) \neq \lambda\\
			1 &\text{ if } z_\rho(\lambda) = 1 \text{ and } \rho_\tau(\lambda) = \lambda.
		\end{cases} \]
		and note that it satisfies the equality
		\[\sum_{\lambda \in B} g(\lambda) = \sum_{\lambda \in A \cup \rho_\tau(A)} z_\rho(\lambda).\]
		Consider the algebraic unit
		\begin{equation}
			\label{eq:productZeta}
			\zeta = \prod_{\lambda \in B} \prod_{\alpha \in \lambda} \Psi(\alpha)^{g(\lambda)}.
		\end{equation}
		Let us prove that $|\zeta| = 1$. Let us write $X = z_{\rho}^{-1}(1)$ and $Y = z_\rho^{-1}(1/2)$. Note that $X \sqcup Y = \Lambda$ and that $X$ is exactly equal to the set of coherent components for which their $\rho$-orbit contains a fixed point of $\rho_\tau$. Since $\Psi$ satisfies \ref{item:thmeigenvalcond3}, there exists a $\theta \in \prod_{\lambda \in \Lambda} \Perm(\lambda)$ such that $\tau \circ \Psi = \Psi \circ r(\rho_\tau) \circ \theta$. We then have
		\begin{align*}
			|\zeta|^2 = \zeta \overline{\zeta} &= \prod_{\lambda \in B} \prod_{\alpha \in \lambda} \Psi(\alpha)^{g(\lambda)} \overline{\Psi(\alpha)}^{g(\lambda)}\\
			&= \prod_{\lambda \in B} \prod_{\alpha \in \lambda} \Psi(\alpha)^{g(\lambda)} (\Psi \circ r(\rho_\tau) \circ \theta)(\alpha)^{g(\lambda)}\\
			&= \prod_{\lambda \in B} \prod_{\alpha \in \lambda} \Psi(\alpha)^{g(\lambda)} (\Psi \circ r(\rho_\tau))(\alpha)^{g(\lambda)}\\
			&= \left( \prod_{\lambda \in B \cap X} \prod_{\alpha \in \lambda} \Psi(\alpha)^{g(\lambda)} (\Psi \circ r(\rho_\tau))(\alpha)^{g(\lambda)} \right) \cdot \left( \prod_{\lambda \in B \cap Y} \prod_{\alpha \in \lambda} \Psi(\alpha)^{g(\lambda)} (\Psi \circ r(\rho_\tau))(\alpha)^{g(\lambda)} \right)\\
			&= \left( \prod_{\substack{\lambda \in B \cup \rho_\tau(B)\\ \lambda \in X}} \prod_{\alpha \in \lambda} \Psi(\alpha)^2 \right) \left(  \prod_{\substack{\lambda \in B \cup \rho_\tau(B)\\ \lambda \in Y}} \prod_{\alpha \in \lambda} \Psi(\alpha) \right)\\
			&= \prod_{\lambda \in B \cup \rho_\tau(B)} \prod_{\alpha \in \lambda} \Psi(\alpha)^{2z_\rho(\lambda)}\\
			&= \prod_{\lambda \in A \cup \rho_\tau(A)} \prod_{\alpha \in \lambda} \Psi(\alpha)^{2z_\rho(\lambda)}.
		\end{align*}
		This last product satisfies all requirements to apply Lemma \ref{lem:productEqualsPmOne} and thus we find that $|\zeta|^2 = \pm 1$ which implies $|\zeta| = 1$. Using that $B$ is connected and that $\Psi$ satisfies condition \ref{item:thmeigenvalcond2}, we must thus have that the number of factors in the product in (\ref{eq:productZeta}) is strictly greater than $c$. The number of factors can be calculated as:
		\begin{align*}
			\sum_{\lambda \in B} \sum_{\alpha \in \lambda} g(\lambda) = \sum_{\lambda \in B \cup \rho_\tau(B)} z_\rho(\lambda) \cdot |\lambda| = \sum_{\lambda \in A \cup \rho_\tau(A)} z_\rho(\lambda) \cdot |\lambda|
		\end{align*}
		which proves that $c < \sum_{\lambda \in A \cup \rho_\tau(A)} z_\rho(\lambda) \cdot |\lambda|$.
		
		Now assume that there exists a $\mu \in \Lambda$ such that $A = \{\mu, \rho_\tau(\mu)\}$ and $\mu \neq \rho_\tau(\mu)$. Note that this implies that $A$ is $\rho$-invariant. Let $\theta \in \prod_{\lambda \in \Lambda} \Perm(\lambda)$ be the permutation such that $\tau \circ \Psi = \Psi \circ r(\rho_\tau) \circ \theta$. Choose a $\gamma \in \mu$ and define $\gamma' = r(\rho_\tau)(\theta(\gamma))$. Since $A$ is connected, we have that $\{\mu, \rho_\tau(\mu)\} \in \overline{E}$ and thus that $\{\gamma', \alpha \} \in E$ for any $\alpha \in \mu$. This implies that the set $\{\gamma'\} \cup (\mu \setminus \{\gamma\})$ is connected in $\G$. Note that
		\begin{align*}
			\left| \Psi(\gamma') \cdot \prod_{\alpha \in \mu \setminus \{\gamma\}} \Psi(\alpha) \right|^2 &= \left| (\Psi \circ r(\rho_\tau) \circ \theta)(\gamma) \cdot \prod_{\alpha \in \mu \setminus \{\gamma\}} \Psi(\alpha) \right|^2\\ 
			&= \left| \overline{\Psi(\gamma)} \cdot \prod_{\alpha \in \mu \setminus \{\gamma\}} \Psi(\alpha) \right|^2\\
			&= \prod_{\alpha \in \mu} \Psi(\alpha) \overline{\Psi(\alpha)}\\
			&= \prod_{\alpha \in \mu} \Psi(\alpha) \prod_{\beta \in \rho_\tau(\mu)} \Psi(\beta)\\
			&= 1
		\end{align*}
		where the last equation follows from applying Lemma \ref{lem:productEqualsPmOne}. Note that $z_\rho(\mu) = z_\rho(\rho_\tau(\mu)) = 1/2$. By condition \ref{item:thmeigenvalcond2} it follows that
		\[c < \left| \{\gamma'\} \cup (\mu \setminus \{\gamma\}) \right| = |\mu| = \sum_{\lambda \in A} z_\rho(\lambda)\cdot|\lambda|,\]
		which completes the proof of the `only if' direction.
		
		Conversely, assume that for any non-empty connected subset $A \subset \Lambda$ such that $A \cup \rho_\tau(A)$ is $\rho$-invariant inequality \ref{eq:inequalityOrbitCondition} holds. In what follows, we will construct a map $\Psi:S \to \Qbar$ which satisfies conditions \ref{item:thmeigenvalcond1}, \ref{item:thmeigenvalcond2} and \ref{item:thmeigenvalcond3} from Theorem \ref{thm:AnosovEigenvalueCond}, proving that $\n_{\rho, c}^\Q$ is Anosov.
		
		Let us list all the $\rho$-orbits as $\Orb_{\rho}(\lambda_1), \ldots, \Orb_{\rho}(\lambda_l)$ such that $\Lambda = \bigsqcup_{i = 1}^l \Orb_{\rho}(\lambda_i)$ and define \[S_i := \bigcup_{\mu \in \Orb_{\rho}(\lambda_i)} \mu.\]
		We first construct maps $\Psi_i:S_i \to \Qbar$ for any $1 \leq i \leq l$. We will write $H_i := \Stab_\rho(\lambda_i)$ to simplify notation. Note that
		\begin{equation}
		\label{eq:degreeStabField}
		\left[\Qbar^{H_i}: \Q\right] = \left[\Gal(\Qbar/\Q) : H_i \right] = \left|\Orb_{\rho}(\lambda_i)\right|
		\end{equation}
		where we used the Galois correspondence for the first equality and the orbit-stabilizer theorem for the second one. Note that $\Qbar^{H_i}$ is a number field and let $s_i$ be the number of real embeddings and $t_i$ the number of conjugated pairs of complex embeddings of $\Qbar^{H_i}$ in $\Qbar$. Lemma \ref{lem:existenceGaloisExtension} now gives us a Galois extension $F_i\big/\,\Qbar^{H_i}$ of degree $m_i := |\lambda_i|$ such that $F_i$ has $s_i \cdot m_i$ real embeddings and $t_i \cdot m_i$ conjugated pairs of complex embeddings. Applying Lemma \ref{lem:existHypUnit} on the field $F_i$, we get an algebraic unit $\xi_i \in F_i$ which satisfies (\ref{eq:chyperbolicity}). We list the cosets of $H_i$ in $\Gal(\Qbar/\Q)$ as $\sigma_{i1} H_i, \ldots , \sigma_{i n_i} H_i$ where $n_i := |\Orb_{\rho}(\lambda_i)| = s_i + 2t_i$. We also order the cosets of $\Gal(\Qbar/F_i)$ in $H_i$ as $\gamma_{i1} \Gal(\Qbar/ F_i), \ldots , \gamma_{i m_i} \Gal(\Qbar/F_i)$ and order the elements of $\lambda_i$ as $\alpha_{i1}, \ldots, \alpha_{i m_i}$. Note that $|S_i| = n_i \cdot m_i$. As a consequence we have parametrized the set $S_i$ by
		\[ S_i = \left\{ r(\rho_{\sigma_{ip}})(\alpha_{iq}) \: \middle| \: 1 \leq p \leq n_i, 1 \leq q \leq m_i \right\}. \]
		We define the map $\Psi_i$ as
		\[ \Psi_i: S_i \to \Qbar: r(\rho_{\sigma_{ip}})(\alpha_{iq}) \mapsto \sigma_{ip} \gamma_{iq} (\xi_i).  \]
		Next, we combine these maps to define for any tuple of integers $\vec{N} = (N_1, \ldots, N_l) \in \Z^l$ the map
		\[\Psi_{\vec{N}}: S\to \Qbar: \alpha \mapsto \Psi_i(\alpha)^{N_i} \quad \text{for} \quad \alpha \in S_i.\]
		Note that since $\xi_i$ is an algebraic unit for any $i$, it follows that the map $\Psi_{\vec{N}}$ satisfies condition \ref{item:thmeigenvalcond1} from Theorem \ref{thm:AnosovEigenvalueCond} for any $\vec{N} \in \Z^l$. One can check that $\Psi_{\vec{N}}$ satisfies condition \ref{item:thmeigenvalcond3} as well for any $\vec{N} \in \Z^l$. To finish the proof we will choose an $\vec{N} \in \Z^l$ such that condition \ref{item:thmeigenvalcond2} holds as well. 
		
		For any weight $e:S \to \Z^{\geq 0}$ define the additive group morphism
		\[ \varphi_e: \Z^l \to \R: \vec{N} \mapsto \log\left|\prod_{\alpha \in S} \Psi_{\vec{N}}(\alpha)^{e(\alpha)}\right| \]
		and define the set of weights
		\[ \mathcal{E}(\G, c) = \{ e:S \to \Z^{\geq 0} \mid \supp(e) \text{ is connected in $\G$ and } 0 < \sum_{\alpha \in S} e(\alpha) \leq c \}. \]
		Let us first prove by contradiction that $\varphi_e$ is not the trivial morphism for any $e \in \mathcal{E}(\G, c)$. Take any $e \in \mathcal{E}(\G, c)$ and assume that $\varphi_e$ is the trivial morphism. Choose any integer $i \in \{1, \ldots, l\}$. Let $\vec{N}_i = (0, \ldots, 0, 1, 0, \ldots, 0)$ be the $l$-tuple with a one on the $i$-th entry and $0$'s elsewhere. If $\varphi_e$ is the trivial morphism, then $\varphi_e(\vec{N}_i) = 0$. This exactly means that
		\[\left| \prod_{\mu \in \Orb_{\rho}(\lambda_i)} \prod_{\alpha \in \mu} \Psi_i(\alpha)^{e(\alpha)} \right| = 1.\]
		By the construction of $\Psi_i$ this gives
		\[ \left|\prod_{p = 1}^{n_i} \prod_{q = 1}^{m_i}  \sigma_{ip}\gamma_{iq} (\xi_i)^{e\left(r(\rho_{\sigma_{ip}})(\alpha_{iq})\right)}\right| = 1. \]
		Note that the restrictions $\sigma_{ip}:\Qbar^{H_i} \to \Qbar$ are exactly all the embeddings of $\Qbar^{H_i}$ into $\Qbar$. We can thus, up to reordering, write the automorphisms $\sigma_{ip}$, $1 \leq p \leq n_i$ as
		\[ \sigma_{i1},\ldots, \sigma_{is_i}, \sigma_{i s_i+1} , \ldots, \sigma_{i s_i +t_i}, \tau \sigma_{i s_i + t_i +1}, \ldots, \tau\sigma_{i s_i + 2t_i} \]
		where $\tau \in \Gal(\Qbar/\Q)$ is the complex conjugation automorphism and the automorphisms $\sigma_{i1}, \ldots, \sigma_{is_i}$ map $\Q^{H_i}$ into $\R \cap \Qbar$. The above equation can thus be rewritten as
		\begin{align*}
		1 &= \left| \prod_{q = 1}^{m_i} \left( \prod_{p = 1}^{s_i}   \sigma_{ip} \gamma_{iq} (\xi_i)^{e\left(r(\rho_{\sigma_{ip}})(\alpha_{iq})\right)} \cdot \prod_{p = 1}^{t_i} \sigma_{ip} \gamma_{iq} (\xi_i)^{e\left(r(\rho_{\sigma_{ip}})(\alpha_{iq})\right)}\cdot  \tau \sigma_{ip} \gamma_{iq} (\xi_i)^{e\left(r(\rho_{\tau \sigma_{ip}})(\alpha_{iq})\right)} \right)\right|\\
		&= \left| \prod_{q = 1}^{m_i} \left( \prod_{p = 1}^{s_i}   \sigma_{ip} \gamma_{iq} (\xi_i)^{e\left(r(\rho_{\sigma_{ip}})(\alpha_{iq})\right)} \cdot \prod_{p = 1}^{t_i} \sigma_{ip} \gamma_{iq} (\xi_i)^{e\left(r(\rho_{\sigma_{ip}})(\alpha_{iq})\right) + e\left(r(\rho_{\tau \sigma_{ip}})(\alpha_{iq})\right)}\right)\right|.
		\end{align*}
		Since the restrictions $\sigma_{ip} \gamma_{iq}: F_i \to \Qbar$ are exactly all the embeddings of $F_i$ into $\Qbar$, we get by the way $\xi_i$ was chosen in $F_i$ using Lemma \ref{lem:existHypUnit} that there is a non-negative integer $k_i \geq 0$ such that
		\begin{alignat*}{2}
		k_i/2 &= e\left(r(\rho_{\sigma_{ip}})(\alpha_{iq})\right) \quad \quad &\forall& 1 \leq p \leq s_i, 1 \leq q \leq m_i\\
		k_i &= e\left(r(\rho_{\sigma_{ip}})(\alpha_{iq})\right) + e\left(r(\rho_{\tau})(r(\rho_{\sigma_{ip}})(\alpha_{iq}))\right) \quad \quad &\forall& s_i+1 \leq p \leq s_i + t_i, 1 \leq q \leq m_i
		\end{alignat*}
		Note that $k_i$ is an even number whenever $s_i > 0$. Doing this for any $i \in \{1, \ldots, l\}$, we get that for the set
		\[A := \{\lambda \in \Lambda \mid \lambda \cap \supp(e) \neq \emptyset\}\]
		it holds that $A \cup \rho_{\tau}(A)$ is $\rho$-invariant. Since $\supp(e)$ is non-empty and connected in $\G$, it is clear that also $A \subset \Lambda$ is non-empty and connected. By the hypothesis of the theorem we must thus have that
		\[ c  < \sum_{\lambda \in A \cup \rho_\tau(A)} z_\rho(\lambda)\cdot |\lambda|. \]
		As we mentioned earlier $k_i$ is even whenever $s_i > 0$ or equivalently whenever $\Qbar^{H_i}$ is not totally imaginary. By Lemma \ref{lem:totallyImaginaryStabField}, we get that $k_i$ is even whenever $z_\rho(\lambda_i) = 1$. Thus, for any $i \in \{1, \ldots, l\}$ with $k_i \neq 0$, we have that $z_\rho(\lambda_i) \leq k_i/2$ and since $z_\rho$ is constant on $\rho$-orbits, that $z_\rho(\mu) \leq k_i/2$ for any $\mu \in \Orb_{\rho}(\lambda_i)$. Since $A \cup \rho_\tau(A)$ is $\rho$-invariant, there exists a subset $I \subset \{1, \ldots, l\}$ such that $A\cup \rho_\tau(A) = \bigcup_{i \in I} \Orb_{\rho}(\lambda_i)$. We then have that
		\[ c < \sum_{\lambda \in A \cup \rho_\tau(A)} z_\rho(\lambda)\cdot |\lambda| \: \leq \: \sum_{i \in I} \sum_{\mu \in \Orb_{\rho}(\lambda_i)} \frac{k_i}{2} \cdot |\lambda| \: = \: \sum_{\alpha \in S} e(\alpha) \]
		which is in contradiction with $\sum_{\alpha \in S} e(\alpha) \leq c$. This proves that $\varphi_e$ is not the trivial morphism for any $e \in \mathcal{E}(\G, c)$ and thus that $\ker(\varphi_e) \neq \Z^l$ for any $e \in \mathcal{E}(\G, c)$. Note that since $\varphi_e$ maps into the torsion-free group $(\R, +)$, the quotient $\Z^l/\ker(\varphi_e)$ is torsion-free as well. Thus for any $e \in \mathcal{E}(\G, c)$ we have that the free abelian group $\ker(\varphi_e)$ has rank at most $l-1$. From this observation, it follows that 
		\[ B := \bigcup_{e \in \mathcal{E}(\G, c)} \ker(\varphi_e) \]
		is not equal to $\Z^l$ since the union is finite. Thus, there exists an element $\vec{M} \in \Z^l \setminus B$. As one can check $\Psi := \Psi_{\vec{M}}$ now satisfies all conditions of Theorem \ref{thm:AnosovEigenvalueCond} which proves what needed to be proven.
	\end{proof}

	If we only consider rational forms of the real Lie algebra $\n^\R_\G$, then $\rho_{\tau}$ is trivial and thus the above theorem simplifies to the following statement.
	
	\begin{cor}[Real version]
		\label{cor:GraphAnosovRatOrbitConditionReal}
		Let $\rho: \Gal(L/\Q) \to \Aut(\overline{\G})$ be an injective morphism, where $L/\Q$ is a real Galois extension. The associated rational form $\n_{\G,\rho}^\Q$ of $\n^{\R}_\G$ is Anosov if and only if for any non-empty connected $\rho$-invariant subset $A \subset \Lambda$, it holds that $c < \sum_{\lambda \in A} |\lambda|$.
	\end{cor}
	
	We can also solely look at the standard rational form $\n_{\G,c}^\Q$ which corresponds to the trivial morphism $\rho:\Gal(\Qbar/\Q) \to \Aut(\overline{\G}):\sigma \mapsto \Id$. A condition for $\n_{\G,c}^\Q$ to be Anosov can then be formulated as follows.
	
	\begin{cor}[Standard rational form]
		\label{cor:standRatForm}
		Let $\G = (S,E)$ be a graph and $\n_{\G,c}^\Q$ the associated rational $c$-step nilpotent Lie algebra. Then $\n_{\G,c}^\Q$ is Anosov if and only if $|\lambda| > 1$ for any $\lambda \in \Lambda$ and for any edge in the quotient graph $e \in \overline{E}$ we have $\sum_{\lambda \in e} |\lambda| > c$. 
	\end{cor}
	
	\begin{remark}[Correction to a result of \cite{main06-1}]
		We like to note that the above corollary is a correction to the characterization stated in \cite[Theorem 4.3.]{main06-1}, which uses the weaker condition of Corollary \ref{cor:standRatForm} only for edges $e$ inside a coherent component $\lambda$ (or thus the resulting loops in the quotient graph). We present an example to show that \cite[Theorem 4.3.]{main06-1} is false. 
		
		Let $\G$ be the graph from Example \ref{ex:completeBiparteGraph} for $n = 2$. The graph and its reduced graph are drawn below.
		
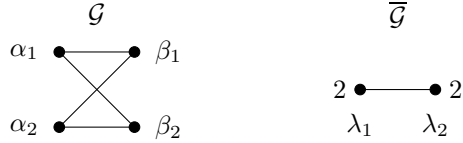
\begin{figure}[H]
			\centering
			\begin{tikzpicture}[every loop/.style={}]
				\draw (0,0) -- (1,1);
				\draw (0,0) -- (1,0);
				\draw (0,1) -- (1,0);
				\draw (0,1) -- (1,1);
				\filldraw [black] (0,0) circle (2pt) node[left = 0.15] {$\alpha_2$};
				\filldraw [black] (1,0) circle (2pt) node[right = 0.15] {$\beta_2$};
				\filldraw [black] (0,1) circle (2pt) node[left = 0.15] {$\alpha_1$};
				\filldraw [black] (1,1) circle (2pt) node[right = 0.15] {$\beta_1$};
				\node at (0.5,1.5) {$\G$};

				\draw (4,0.5) -- (5,0.5);
				\filldraw [black] (4,0.5) circle (2pt) node[anchor = east] {2\,} node[below = 2mm] {$\lambda_{1}$};
				\filldraw [black] (5,0.5) circle (2pt) node[anchor = west] {\,2} node[below = 2mm] {$\lambda_{2}$};
				\node at (4.5,1.5) {$\overline{\G}$};
			\end{tikzpicture}
		\end{figure}
		\noindent The false result in \cite{main06-1} claims that $\n_{\G, c}^\Q$ is Anosov for any integer $c > 1$, while in fact this Lie algebra is only Anosov for $c < 4$. The problem in the proof of \cite[Theorem 4.3.]{main06-1} lies with the eigenvectors arising from Lie brackets of vertices lying in different coherent components. In this example, the Lie bracket $[\alpha_1, [ \beta_1, [ \alpha_2, \beta_2]]]$ is non-zero in $\n_{\G, c}^\Q$ with $c \geq 4$ and will be an eigenvector with eigenvalue $\pm 1$ for any vertex-diagonal integer-like automorphism of $\n_{\G, c}^\Q$. This essentially proves $\n_{\G, c}^\Q$ is not Anosov for $c \geq 4$, but a detailed proof is provided by our main results, in particular Corollary \ref{cor:standRatForm}.
	\end{remark}
	
	The above two corollaries show that from all rational forms in $\n^{\R}_{\G, c}$ that can be Anosov, the standard one $\n^{\Q}_{\G, c}$ leads to the strongest condition.
	
	\begin{cor}
		\label{cor:standRatAnosovThenRatInRealAnosov}
		Let $\G$ be a simple undirected graph and $\n^\Q_{\G, c}$ the associated $c$-step nilpotent rational Lie algebra. If $\n^\Q_{\G,c}$ admits an Anosov automorphism, then so does any other rational form of the real Lie algebra $\n^{\R}_{\G,c}$.
	\end{cor}
	
	\begin{remark}
		Note that Corollary \ref{cor:standRatAnosovThenRatInRealAnosov} does not generalize to the rational forms of the complex Lie algebra $\n_{\G,c}^\C$. Example \ref{ex:twoCopiesHeisenbergComplement} will illustrate this. 
	\end{remark}

	\section{Examples and applications}
	\label{sec:applications}
	
	In this section we illustrate how easy it is to apply our main result for some families of graphs. In particular, certain classifications in low dimension, as given in \cite{lw09-1} are an immediate consequence of Theorem \ref{thm:AnosovOrbitCondInjective}. First, let us mention a classical result by Erd\"os and Renyi stating that almost every graph has a trivial automorphism group, see \cite{er63-1}. In particular, these graphs have coherent components of size $1$. Using Theorem \ref{thm:injectiveVersionClassificationRationalForms} and Theorem \ref{thm:AnosovOrbitCondInjective} we thus immediately get the following result.
	
	\begin{theorem}
		Fixing an integer $c > 1$, the Lie algebra $\n_{\G, c}^\C$ has an (up to $\Q$-isomorphism) unique rational form which is not Anosov for almost all graphs $\G$.
	\end{theorem}
	
	\noindent This shows that having an Anosov rational form is a rare condition for Lie algebras associated to graphs and raises the question whether a similar statement holds for nilpotent Lie algebras in general.\\
	
	Let us now apply Theorem \ref{thm:AnosovOrbitCond} to certain classes of simple undirected graphs, a first of which is \textit{trees}. A \textit{tree} is a graph in which any two vertices are connected by exactly one path. Equivalently, a graph is a tree if and only if it is connected and has no cycles. For two vertices $\alpha,\beta$ in a connected graph $\G = (S, E)$, let $d(\alpha, \beta)$ denote the distance between $\alpha$ and $\beta$, given by the minimal number of edges needed to go from $\alpha$ to $\beta$ in the graph. The \textit{eccentricity} $e(\alpha)$ of a vertex $\alpha$ is defined as $e(\alpha) = \max\{ d(\alpha,\beta) \mid \beta \in S \}$. The \textit{center} of $\G$ is then defined as the set of vertices of $\G$ which have minimal eccentricity. It is a standard result that a tree has a center consisting of either one or two vertices. To illustrate this, two trees with their center are draw in Figure \ref{fig:centerTree1} and \ref{fig:centerTree2}.
	\begin{figure}[H]
		\centering
		\begin{minipage}{0.45\textwidth}
			\centering
			\begin{tikzpicture}
			\centering
			\draw (0,0) -- (1,0);
			\draw (1,0) -- (2,0);
			\filldraw [black] (0,0) circle (2pt);
			\filldraw [red] (1,0) circle (2pt);
			\filldraw [black] (2,0) circle (2pt);
			\filldraw [white] (1,-0.7) circle (2pt);
			\filldraw [white] (1,0.7) circle (2pt);
			\end{tikzpicture}
			\caption{A tree with a center consisting of one vertex, drawn in red.}
			\label{fig:centerTree1}
		\end{minipage} \hspace{5mm}
		\begin{minipage}{0.45\textwidth}
			\centering
			\begin{tikzpicture}
			\centering
			\draw [red] (0,0) -- (1,0);
			\draw (1,0) -- (1.7,0.7);
			\draw (1,0) -- (1.7,-0.7);
			\draw (0,0) -- (-0.7,0.7);
			\draw (0,0) -- (-0.7,-0.7);
			\filldraw [red] (0,0) circle (2pt);
			\filldraw [red] (1,0) circle (2pt);
			\filldraw [black] (1.7,0.7) circle (2pt);
			\filldraw [black] (1.7,-0.7) circle (2pt);
			\filldraw [black] (-0.7,0.7) circle (2pt);
			\filldraw [black] (-0.7,-0.7) circle (2pt);
			\end{tikzpicture}
			\caption{A tree with a center consisting of two adjacent vertices, drawn in red.}
			\label{fig:centerTree2}
		\end{minipage}
	\end{figure}

	\begin{prop}
		If $\G$ is a tree, then $\n^{\C}_{\G, c}$ has no Anosov rational forms for any $c > 1$.
	\end{prop}
	\begin{proof}
		Take an arbitrary continuous morphism $\rho:\Gal(L/\Q) \to \Aut(\overline{\G})$. Let $C \subset S$ be the center of $\G = (S, E)$. Note that vertices in the same coherent component can be mapped onto each other by an automorphism of $\G$ and must thus have the same eccentricity. By consequence, the center is a union of coherent components invariant under $\Aut(\G)$. Since $\G$ is a tree, we are left with three cases:
		\begin{itemize}
			\item $|C| = 1$. Note that $C$ is itself a coherent component and must be preserved under any automorphism of $\overline{\G}$, implying that $\{C\}$ is $\rho$-invariant non-empty and connected. Since $|C| = 1 < c$, we get by Theorem \ref{thm:AnosovOrbitCond} that $\n_{\rho, c}^\Q$ is not Anosov.
			\item $|C| = 2$ and $C$ is a coherent component. Since the coherent component $C$ is preserved under any automorphism of $\overline{\G}$, we get that $\{C\}$ is $\rho$-invariant and non-empty. If the center of a tree contains two vertices, then they are adjacent. By consequence $\{C\}$ is connected as well. Since $|C| = 2 \leq c$, Theorem \ref{thm:AnosovOrbitCondInjective} tells us that $\n_{\rho, c}^\Q$ is not Anosov.
			\item $|C| = 2$ and $C$ is a union of two disjoint coherent components. Let us write $C = \lambda \cup \mu$ with $\lambda$ and $\mu$ disjoint coherent components. Note that $\{\lambda, \mu\} \in \overline{E}$ and thus that $\{\lambda, \mu\}$ is a non-empty connected set of coherent components. Since the center is preserved under any automorphism of $\overline{\G}$, we get that $\{\lambda, \mu\}$ is $\rho$-invariant as well. Depending on whether $\rho_\tau$ fixes $\lambda$ or not, we get that $z_\rho(\lambda) |\lambda| + z_\rho(\mu) |\mu|$ is equal to $1$ or $2$, respectively. In either case, Theorem \ref{thm:AnosovOrbitCond} tells us that $\n_{\rho, c}^\Q$ is not Anosov.
		\end{itemize}
		This concludes the proof.
	\end{proof}
	
	As a second class, let us consider the \textit{cycle graphs}. These graphs can be considered as the simplest graphs which are not trees and therefore the natural class to consider next. The cycle graph of size $n$ is given by vertices $S = \{1, \ldots, n\}$ and edges $E = \{ \{1,2\}, \{2, 3\}, \ldots, \{ n-1, n \}, \{n, 1\} \}$. If $n \geq 5$, then the coherent components are all singletons as illustrated below in Figure \ref{fig:cyclegraph6} for $n = 6$. It follows that for $n \geq 5$, the automorphism group of the quotient graph is isomorphic to the dihedral group of order $2n$. Let $a$ be a generator of the rotation subgroup of $\Aut(\overline{\G})$ and $b$ a reflection of $\Aut(\overline{\G})$. Then $\Aut(\overline{\G}) = \{\Id, a , \ldots, a^{n-1}, b, ab, \ldots a^{n-1}b\}$. Let us call a rational form of $\n_\G^\R$ of \textit{reflection type} if the corresponding representation $\Gal(L/\Q) \to \Aut(\overline{\G})$ has image $\{\Id, a^i b\}$ for some $1 \leq i \leq n$. Then using Corollary \ref{cor:GraphAnosovRatOrbitConditionReal}, it is not hard to prove following statement.
	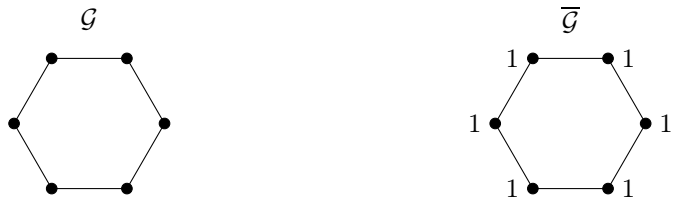
\begin{figure}[H]
		\centering
		\begin{subfigure}[b]{0.35\textwidth}
			\centering
			\begin{tikzpicture}
				\draw (0,0) -- (1,0);
				\draw (1,0) -- (1.5, 0.866025);
				\draw (1.5, 0.866025) -- (1, 2*0.866025);
				\draw (0,0) -- (-0.5, 0.866025);
				\draw (-0.5, 0.866025) -- (0, 2*0.866025);
				\draw (0, 2*0.866025) -- (1, 2*0.866025);
				\filldraw [black] (0,0) circle (2pt);
				\filldraw [black] (1,0) circle (2pt);
				\filldraw [black] (1.5, 0.866025) circle (2pt);
				\filldraw [black] (-0.5, 0.866025) circle (2pt);
				\filldraw [black] (0, 2*0.866025) circle (2pt);
				\filldraw [black] (1, 2*0.866025) circle (2pt);
				\node at (0.5,2.25) {$\mathcal{G}$};
				\filldraw [white] (0.5, -0.5) circle (2pt);
			\end{tikzpicture}
		\end{subfigure}\hspace{5mm}
		\begin{subfigure}[b]{0.35\textwidth}
			\centering
			\begin{tikzpicture}
				\draw (0,0) -- (1,0);
				\draw (1,0) -- (1.5, 0.866025);
				\draw (1.5, 0.866025) -- (1, 2*0.866025);
				\draw (0,0) -- (-0.5, 0.866025);
				\draw (-0.5, 0.866025) -- (0, 2*0.866025);
				\draw (0, 2*0.866025) -- (1, 2*0.866025);
				\filldraw [black] (0,0) circle (2pt) node[anchor = east] {1\,};
				\filldraw [black] (1,0) circle (2pt) node[anchor = west] {\,1};
				\filldraw [black] (1.5, 0.866025) circle (2pt) node[anchor = west] {\,1};
				\filldraw [black] (-0.5, 0.866025) circle (2pt) node[anchor = east] {1\,};
				\filldraw [black] (0, 2*0.866025) circle (2pt) node[anchor = east] {1\,};
				\filldraw [black] (1, 2*0.866025) circle (2pt) node[anchor = west] {\,1};
				\node at (0.5,2.25) {$\overline{\G}$};
				\filldraw [white] (0.5, -0.5) circle (2pt);
			\end{tikzpicture}
		\end{subfigure}
		\caption{The cycle graph on 6 vertices and its quotient graph.}
		\label{fig:cyclegraph6}
	\end{figure}
	\begin{prop}
		Let $\G$ be a cycle graph of size $n \geq 5$ and take any $c > 1$. The standard rational form $\n^\Q_{\G,c}$ and all reflection-type rational forms of $\n^\R_{\G,c}$ are not Anosov. The other rational forms of $\n^\R_{\G,c}$ are Anosov if and only if $n > c$.
	\end{prop}

	\begin{proof}
		Since the coherent components of $\G$ are all singletons, it follows by Corollary \ref{cor:standRatForm} that $\n_{\G,c}^\Q$ is not Anosov. If $\n^\Q_{\rho, c}$ is a reflection-type rational form of $\n_{\G,c}^\R$, then after possibly conjugating the $\rho$-action by an automorphism of the quotient graph, either $\{\{1\}\}$ or $\{\{1\}, \{2\}\}$ is $\rho$-invariant. In any case Corollary \ref{cor:GraphAnosovRatOrbitConditionReal} tells us that $\n^\Q_{\rho, c}$ is not Anosov. If $\n^\Q_{\rho, c}$ is any other rational form of $\n_{\G, c}^\R$, then the image of $\rho$ must contain a non-trivial rotation. Now suppose $A \subset \Lambda$ is a non-empty connected $\rho$-invariant subset of coherent components. Since $A$ must be connected, we get that up to possibly conjugating the action of $\rho$ by an automorphism of the quotient graph, that $A = \{ \{1\}, \ldots , \{k\}\}$ for some $1 \leq k \leq n$. Using that $A$ is also $\rho$-invariant, we know it must be preserved under the non-trivial rotation. This is only possible if $k = n$ implying that $A = \Lambda = \{\{1\}, \ldots, \{n\}\}$ and $\sum_{\lambda \in A} |\lambda| = n$. By Corollary \ref{cor:GraphAnosovRatOrbitConditionReal} we get that $\n_{\rho, c}^\R$ is Anosov if and only if $n > c$.
	\end{proof}

	\begin{prop}
		If $\G$ is a graph for which there is a non-negative integer $k \geq 0$ such that $\G$ has a unique vertex of degree $k$, then $\n_{\G, c}^\C$ has no Anosov rational forms for any $c>1$.
	\end{prop}	
	
	\begin{proof}
		Note that all vertices in a coherent component have the same degree. By consequence, if $\G$ has a vertex $\alpha$ for which there are no other vertices of the same degree, then $\{\alpha\}$ is a coherent component of $\G$. Since graph automorphisms must preserve the degree of a vertex, it follows that $\alpha$ is fixed under any graph automorphism and by consequence the coherent component $\{\alpha\}$ is fixed under any automorphism of the quotient graph. It follows that for an arbitrary morphism $\rho:\Gal(L/\Q) \to \Aut(\overline{\G})$, the set $\{\{\alpha\}\}$ is a non-empty connected $\rho$-invariant set. Since $|\{\alpha\}| = 1 < c$, Theorem \ref{thm:AnosovOrbitCond} tell us that $\n_{\rho, c}^\Q$ is not Anosov.
	\end{proof}

	The following examples show that our methods can be used to simplify certain classifications of Anosov Lie algebras and to extend them.
	
	\begin{example}[Direct sum of two free nilpotent Lie algebras of same rank and nilpotency class]
		\label{ex:twoCopiesHeisenberg}
		Consider the graph $\G_n = (S_n, E_n)$ defined by $S_n = \{ \alpha_1, \ldots, \alpha_n, \beta_1, \ldots, \beta_n\}$ and $E_n = \{ \{\alpha_i, \alpha_j \} \mid 1 \leq i < j \leq n\} \cup \{\{\beta_i, \beta_j\} \mid 1 \leq i < j \leq n\}$. The set of coherent components is then given by $\Lambda = \{ \lambda_1, \lambda_2 \}$ with $\lambda_1 = \{\alpha_1, \ldots, \alpha_n\}$ and $\lambda_2 = \{\beta_1, \ldots, \beta_n\}$. A figure of the quotient graph is given below.
		
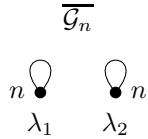
\begin{figure}[H]
			\centering
			\begin{tikzpicture}[every loop/.style={}]
				\filldraw [black] (-0.5,0.5) circle (2pt) node[anchor = east] {$n\:$} node[below = 0.15] {$\lambda_1$};
				\filldraw [black] (0.5,0.5) circle (2pt) node[anchor = west] {$\:n$} node[below = 0.15] {$\lambda_2$};
				\draw[cm={1.5 ,0 ,0 ,1.5 ,(-0.5,0.5)}] (0,0)  to[in=50,out=130, loop] (0,0);
				\draw[cm={1.5 ,0 ,0 ,1.5 ,(0.5,0.5)}] (0,0)  to[in=50,out=130, loop] (0,0);
				\node at (0,1.5) {$\overline{\G_n}$};
			\end{tikzpicture}
		\end{figure}
		\noindent The Lie algebra $\n_{\G_n, c}^\C$ is isomorphic to the Lie algebra direct sum of two free $c$-step nilpotent Lie algebras of rank $n$. Let $\varphi \in \Aut(\overline{\G_n})$ denote the only non-trivial automorphism of $\overline{\G_n}$. For any non-zero square-free integer $d \neq 1$, let $\sigma_d$ denote the only non-trivial automorphism in $\Gal(\Q(\sqrt{d})/\Q)$. If $d = 1$, then $\Gal(\Q(\sqrt{d})/\Q)$ is trivial and we let $\sigma_1$ denote the trivial automorphism. All injective group morphisms from a Galois group over $\Q$ to $\Aut(\overline{\G_n})$ are then given by
		\[ \rho_d:\Gal(\Q(\sqrt{d})/\Q) \to \Aut(\overline{\G_n}): \sigma_d \mapsto \varphi\]
		for some non-zero square-free integer $d$. All rational forms of $\n_{\G_n, c}^\C$ can thus be written as
		\[\n_{d, n, c}^\Q := \n_{\rho_d, c}^\Q \subset \n_{\G_n, c}^{\Q(\sqrt{d})}\]
		for a non-zero square-free integer $d$. Moreover, for non-zero square-free integers $d_1, d_2$ it holds that $\n_{d_1, n , c}^\Q \cong \n_{d_2, n , c}^\Q$ if and only if $d_1 = d_2$. For $d = 1$, $\rho_d$ is the trivial morphism and we retrieve the standard rational form $\n_{1, n , c}^\Q \cong \n_{\G_n, c}^\Q$. From Theorem \ref{thm:AnosovOrbitCond}, we can now easily see that for a square-free non-zero integer $d$: 
		\[ \boxed{\n_{d, n, c}^\Q \text{ is Anosov } \Leftrightarrow d > 1 \vee (d \leq 1 \wedge c < n).} \]
		This result was already known for $c = 2$ and $n = 2$ from the classification of Anosov Lie algebras of dimension $6$, where the Lie algebra $\n_{k,2,2}^\Q$ was denoted by $\n_k^\Q$ (see \cite[Example 2.7. and Theorem 4.2.]{lw09-1}).
	\end{example}
	
	\begin{definition}
		Let $\G = (S,E)$ be a simple undirected graph. The graph $\G^\ast := (S,E^\ast)$ with
		\[E^\ast = \{ \{ \alpha,\beta \} \mid \alpha,\beta \in S, \alpha \neq \beta, \{\alpha,\beta\} \notin E \} \]
		is called the \textit{complement graph of $\G$}.
	\end{definition}

	Note that the coherent components of a simple undirected graph and its complement graph coincide. Moreover, we have that $\Aut(\overline{\G}) = \Aut(\overline{\G^\ast})$. This being said, let us look at the Lie algebra associated with the complement graph of the one from Example \ref{ex:twoCopiesHeisenberg}.
	
	\begin{example}[Free nilpotent sum of two abelian Lie algebras of same dimension]
		\label{ex:twoCopiesHeisenbergComplement}
		Let $\G_n = (S_n,E_n)$ be the graph from Example \ref{ex:twoCopiesHeisenberg} and $\G_n^\ast$ its complement graph. Note that $\G_n^\ast$ is also the cycle graph of size $4$. The quotient graph $\overline{\G_n^\ast}$ is drawn below.
		
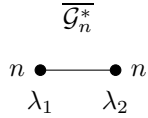
\begin{figure}[H]
			\centering
			\begin{tikzpicture}[every loop/.style={}]
				\draw (-0.5,0.5) -- (0.5,0.5);
				\filldraw [black] (-0.5,0.5) circle (2pt) node[anchor = east] {$n\:$} node[below = 0.15] {$\lambda_1$};
				\filldraw [black] (0.5,0.5) circle (2pt) node[anchor = west] {$\:n$} node[below = 0.15] {$\lambda_2$};
				\node at (0,1.2) {$\overline{\G_n^\ast}$};
			\end{tikzpicture}
		\end{figure}
		\noindent The Lie algebra $\n_{\G_n, c}^L$ is isomorphic to the free $c$-step nilpotent sum of two abelian Lie algebras of dimension $n$. Let $\varphi \in \Aut(\overline{\G_n})$ denote the only non-trivial automorphism of $\overline{\G_n}$. Since $\Aut(\overline{\G_n}) = \Aut(\overline{\G_n^\ast})$, it follows that all injective group morphisms from finite Galois groups over $\Q$ into $\Aut(\overline{\G_n^\ast})$ are given by the same morphisms $\rho_d$ for $d$ a non-zero square-free integer as defined in Example \ref{ex:twoCopiesHeisenberg}. As a consequence all rational forms of $\n_{\G_n^\ast, c}^\C$ are given by
		\[\n_{d, n, c}^{\Q, \ast} := \n_{\rho_d, c}^\Q \subset \n_{\G_n^\ast, c}^{\Q(\sqrt{d})}\]
		for a non-zero square-free integer $d$. Moreover, for non-zero square-free integers $d_1, d_2$ it holds that $\n_{d_1, n , c}^{\Q, \ast} \cong \n_{d_2, n , c}^{\Q, \ast}$ if and only if $d_1 = d_2$. For $d = 1$ we retrieve the standard rational form $\n_{1, n , c}^{\Q, \ast} \cong \n_{\G_n^\ast, c}^\Q$. From Theorem \ref{thm:AnosovOrbitCond}, we can now easily see that for a square-free non-zero integer $d$:
		\[ \boxed{\n_{d, n, c}^{\Q, \ast} \text{ is Anosov } \Leftrightarrow (d \geq 1 \wedge c < 2n) \vee (d < 1 \wedge c < n).} \]
		This result was already known for $c = 2$, $n = 2$ and $d \geq 1$ from the classification of real Anosov Lie algebras of dimension  $8$, where the Lie algebra $\n_{k,2,2}^{\Q,\ast}$ was denoted by $\mathfrak{h}_k^\Q$ (see \cite[Theorem 4.2.]{lw09-1}).
	\end{example}
	
	When studying Anosov automorphisms on rational Lie algebras in low dimensions, one observes that in all known cases an Anosov Lie algebra always has an Anosov automorphism with only real eigenvalues. The results in this paper allow us to present an example where this is no longer the case, so where every Anosov automorphism has non-real eigenvalues.
	
	\begin{example}[Anosov rational form which does not admit Anosov automorphism with real eigenvalues]
		\noindent Let $\G = (S, E)$ be the cycle graph on $6$ vertices as drawn in Figure \ref{fig:cyclegraph6}. Let us write the vertices and edges as $S = \{ \alpha_1, \ldots, \alpha_6 \}$ and $E = \{ \{ \alpha_{1}, \alpha_{2} \}, \ldots \{\alpha_5, \alpha_6\}, \{\alpha_6, \alpha_1 \}\}$. The coherent components of $\G$ are then simply all the singletons $\Lambda = \{ \lambda_i := \{\alpha_i\} \mid 1\leq i \leq 6 \}$. By consequence there is a natural bijection $h:S \to \Lambda: \alpha \mapsto \{\alpha\}$ which gives us a splitting morphism $r:\Aut(\overline{\G}) \to \Aut(\G): \varphi \mapsto h^{-1} \circ \varphi \circ h$. Let us write $\Aut(\overline{\G}) = \{1, a, \ldots a^{5}, b, ab, \ldots, a^5 b\}$ where $a$ and $b$ are defined by $a(\lambda_1) = \lambda_{2}$, $a(\lambda_2) = \lambda_3$, $b(\lambda_1) = \lambda_1$ and $b(\lambda_2) = \lambda_6$. Thus $a$ is a generator for the rotations and $b$ is a reflection, like in our general discussion of cycle graphs.
		
		Now let $L$ be the splitting field of the polynomial $X^3 - 2$ over $\Q$. The roots of this polynomial are given by $\sqrt[3]{2}$, $\omega \sqrt[3]{2}$ and $\overline{\omega}\sqrt[3]{2}$ where $\omega = e^{i\frac{2\pi}{3}}$. The Galois group $\Gal(L/\Q)$ is generated by the elements $\sigma$ and $\tau$, defined by
		\begin{align*} 
			\sigma(\sqrt[3]{2}) = \omega\sqrt[3]{2}, \quad \sigma(\omega\sqrt[3]{2}) = \overline{\omega}\sqrt[3]{2}, \quad  \sigma(\overline{\omega}\sqrt[3]{2}) = \sqrt[3]{2},\\
			\tau(\sqrt[3]{2}) = \sqrt[3]{2}, \quad \tau(\omega\sqrt[3]{2}) = \overline{\omega}\sqrt[3]{2}, \quad \tau(\overline{\omega}\sqrt[3]{2}) = \omega\sqrt[3]{2}.
		\end{align*}
		Note that $\tau$ is just the complex conjugation automorphism on $L$ and that $\Gal(L/\Q)$ is isomorphic to the dihedral group of order $6$. It follows that we have an injective group morphism
		\[\rho:\Gal(L/\Q) \to \Aut(\overline{\G}): \sigma \mapsto a^2, \tau \mapsto b\]
		with corresponding rational form $\n^\Q_{\rho,2}$ of $\n_{\G, 2}^\C$. Using Theorem \ref{thm:AnosovOrbitCondInjective}, it is straightforward to verify that $\n_{\rho, 2}^\Q$ is Anosov. Indeed, $z_\rho$ takes only the value 1 on $\Lambda$ and for any non-empty connected $A \subset \Lambda$ for which $A \cup \rho_\tau(A)$ is $\rho$-invariant, we have that $A \cup \rho_\tau(A) = \Lambda$. 
		
		For the convenience of the reader, we construct an explicit basis for the rational Lie algebra $\n_{\G, \rho}^\Q$ and compute the structure constants. We define the vectors $\beta_{i} = [\alpha_{i}, \alpha_{i+1}]$ for $1 \leq i \leq 5$ and $\beta_6 = [\alpha_6, \alpha_1]$. The following elements form a basis for $\n_{\G, \rho}^\Q$:		
		\begin{align*}
			X_i &= \left( \sqrt[3]{2} \right)^i \alpha_1 + \left( \omega\sqrt[3]{2} \right)^i \alpha_3 + \left( \overline{\omega}\sqrt[3]{2} \right)^i \alpha_5\\
			Y_i &= \left( \overline{\omega}\sqrt[3]{2} \right)^i \alpha_2 + \left( \sqrt[3]{2} \right)^i \alpha_4 + \left( \omega\sqrt[3]{2} \right)^i \alpha_6\\
			Z_i &= \left(\left( \sqrt[3]{2} \right)^i \beta_1 + \left( \omega\sqrt[3]{2} \right)^i \beta_3 + \left( \overline{\omega}\sqrt[3]{2} \right)^i \beta_5\right)\\
			&\quad \quad \quad - \left( \left(\omega \sqrt[3]{2}\right)^i \beta_2 + \left(\overline{\omega}\sqrt[3]{2}\right)^i \beta_4 + \left(\sqrt[3]{2}\right)^i \beta_6\right)\\
			W_i &= \omega\left(\left( \sqrt[3]{2} \right)^i \beta_1 + \left( \omega\sqrt[3]{2} \right)^i \beta_3 + \left( \overline{\omega}\sqrt[3]{2} \right)^i \beta_5\right)\\
			&\quad \quad \quad - \overline{\omega}\left( \left(\omega \sqrt[3]{2}\right)^i \beta_2 + \left(\overline{\omega}\sqrt[3]{2}\right)^i \beta_4 + \left(\sqrt[3]{2}\right)^i \beta_6\right)
		\end{align*}
		where $0 \leq i \leq 2$. In this basis, the Lie bracket of $\n^\Q_{\G, \rho}$ is given by the following relations
		\begin{alignat*}{3}
			[X_0, Y_0] &= Z_0 & \quad \quad [X_1, Y_0] &= Z_1 & \quad \quad [X_2, Y_0] &= Z_2\\
			[X_0, Y_1] &= -Z_1 - W_1 & \quad \quad [X_1, Y_1] &= -Z_2 - W_2 & \quad \quad [X_2, Y_1] &= -2Z_0 - 2W_0\\
			[X_0, Y_2] &= W_2 & \quad \quad [X_1, Y_2] &= 2W_0 & \quad \quad [X_2, Y_2] &= 2W_1.			
		\end{alignat*}
		Let us prove by contradiction that $\n_{\G, \rho}$ does not admit an Anosov automorphism with real eigenvalues. So, assume $f:\n_{\G, \rho}^\Q \to \n_{\G, \rho}^\Q$ is an Anosov automorphism with real eigenvalues. From Theorem \ref{thm:AnosovEigenvalueCond} and its proof, we know that there exists a map $\Psi:S \to \Qbar$, which satisfies the properties formulated in that theorem and such that the algebraic units in the image of $\Psi$ are eigenvalues of $f$. By the assumption, these eigenvalues are all real. Let $\overline{\rho}:\Gal(\Qbar/\Q) \to \Aut(\overline{\G}): \gamma \mapsto \rho(\gamma|_L)$ be the extended morphism. Note that since all coherent components are singletons, the group $\prod_{\lambda \in \Lambda} \Perm(\lambda)$ is trivial and thus condition \ref{item:thmeigenvalcond3} on $\Psi$ becomes $\forall \gamma \in \Gal(\Qbar/\Q): \gamma \circ \Psi = \Psi \circ r(\overline{\rho}_\gamma)$. Write $\xi := \Psi(\alpha_1)$. Now let $\gamma$ be any element in $\Stab_{\overline{\rho}}(\lambda_1)$. It follows that $\gamma(\xi) = (\gamma \circ \Psi)(\alpha_{1}) = (\Psi \circ r(\overline{\rho}_{\gamma}))(\alpha_{1}) = \Psi(\alpha_{1}) = \xi$. By consequence $\xi$ is fixed under $\Stab_{\overline{\rho}}(\lambda_{1})$ and in particular under $\ker(\overline{\rho})$. Note that $\Qbar^{\ker(\overline{\rho})}$ is exactly equal to $L$ and thus that $\xi \in L$. Now let $\overline{\sigma}$ be an extension of the field automorphism $\sigma$ to $\Qbar$ and note that $\sigma(\xi) = \overline{\sigma}(\xi) = (\overline{\sigma} \circ \Psi)(\alpha_{1}) = (\Psi \circ r(\overline{\rho}_{\overline{\sigma}}))(\alpha_1) = \Psi(\alpha_3)$ is also an eigenvalue of $f$. Since $\xi$ and $\sigma(\xi)$ are both real, we have $\tau(\xi) = \xi$ and $\tau\sigma(\xi) = \sigma(\xi)$. This implies $\sigma(\xi) = \tau\sigma(\xi) = \sigma^2 \tau(\xi) = \sigma^2(\xi)$ and thus that $\sigma(\xi) = \xi$. By consequence $\xi$ is an element of $L$, fixed by $\Gal(L/\Q)$, which in turn tells us that $\xi \in \Q$. The only algebraic units in $\Q$ are $1$ and $-1$ which are not hyperbolic. This gives us the contradiction.
	\end{example}

	\bibliography{ref}
	\bibliographystyle{plain}	
\end{document}